\documentclass[preprint,10pt]{elsarticle}
\usepackage{amsfonts,amssymb,amsmath}
\usepackage{graphicx,color}
\usepackage{latexsym}
\usepackage{amsthm}
\usepackage{graphics}
\usepackage{subfigure}
\usepackage{epstopdf}
\usepackage{natbib}
\usepackage{float}
\usepackage{bm}
\usepackage{indentfirst}
\usepackage[colorlinks,
            linkcolor=red,
            anchorcolor=blue,
            citecolor=green
            ]{hyperref}
\newtheorem{prop}{Proposition}[section]

\newtheorem{remark}{Remark}[section]

\def\bU{{\bf{U}}}
\def\bV{{\bf{V}}}
\def\bG{{\bf{G}}}
\def\bS{{\bf{S}}}
\def\bF{{\bf{F}}}
\def\btU{{\bf{\widetilde{U}}}}
\def\btS{{\bf{\widetilde{S}}}}

\journal{arXiv}
\begin{document}

\begin{frontmatter}
\title{A CDG-FE method for the two-dimensional Green-Naghdi model with the enhanced dispersive property}
\author[a]{Maojun Li}
\ead{limj@cqu.edu.cn}
\author[a]{Liwei Xu\corref{cor1}}
\ead{xul@uestc.edu.cn}
\cortext[cor1]{Corresponding author.}
\author[b]{Yongping Cheng}
\ead{cyp@cqu.edu.cn}

\address[a]{School of Mathematical Sciences, University of Electronic Science and Technology of China, Sichuan, 611731, P.R. China }
\address[b]{College of Mathematics and Statistics, Chongqing University, Chongqing, 401331, P.R. China}
\begin{abstract}
In this work, we investigate numerical solutions of the two-dimensional shallow water wave using a fully nonlinear Green-Naghdi model with an improved dispersive effect.  For the purpose of numerics, the Green-Naghdi model is rewritten into a formulation coupling a pseudo-conservative system and a set of pseudo-elliptic equations. Since the pseudo-conservative system is no longer hyperbolic and its Riemann problem can only be approximately solved, we consider the utilization of the central discontinuous Galerkin method which possesses an important feature of needlessness of Riemann solvers. Meanwhile, the stationary elliptic part will be solved  using the finite element method.  Both the well-balanced  and the positivity-preserving features which are highly desirable in the simulation of the shallow water wave will be embedded into the proposed numerical scheme. The accuracy and efficiency of
 the numerical model and method will be illustrated through numerical tests.
\end{abstract}


\end{frontmatter}

\section{Introduction}
\label{sec:1}
In an incompressible and inviscid fluid, the propagation of surface waves is governed by the Euler equation  with nonlinear boundary conditions at the free surface and the bottom. In its full generality, this problem is very complicated to be solved, both mathematically and numerically. Usually, simplified models have been derived to describe the behavior of the solution in some physical specific regimes, such as  the nonlinear shallow water equations,  the Korteweg-de Vries equations, the Boussinesq type models and the Green-Naghdi models, and to name a few.

The nonlinear shallow water equations (also called Saint-Venant equations) are a set of hyperbolic partial differential equations (or parabolic if viscous shear is considered) that describe the flow below a pressure surface in a fluid. They can model the propagation of strongly nonlinear waves up to breaking and run-up in near-shore zones. However, they fail to properly describe wave propagation in deep water or wave shoaling because they do not incorporate frequency dispersion. The Boussinesq systems carry weak frequency dispersion but are typically restricted to small amplitude waves with relatively weak nonlinearity. Extensions of the shallow water equations that incorporate frequency dispersion can be traced back to Serre (\cite{Serre1953}) who derived a one-dimensional (1D) system of equations for fully nonlinear weakly dispersive waves over the flat bottom in 1953. In 1976, Green and Naghdi (\cite{Green1976}) presented the two-dimensional (2D) counterpart of these equations for wave propagation over variable bottom topography. The Green-Naghdi model is a class of fully nonlinear weakly dispersive shallow water wave equations, including dispersive effects and supporting traveling solitary wave solutions. Therefore, it can simulate the long-time propagation of solitary waves with relatively large amplitude.

Due to the significance of the Green-Naghdi model, there has been increasing interest in the numerical simulation of the Green-Naghdi model  in the past decade. A fourth-order compact finite volume scheme was proposed for solving the fully nonlinear and weakly dispersive Boussinesq type equations (\cite{Cien2006,Cien2007}). A hybrid numerical method using a Godunov type scheme was presented to solve the Green-Naghdi model over the flat bottom (\cite{MGH10}). A pseudo-spectral algorithm was developed for the solution of the rotating Green-Naghdi shallow water equations (\cite{Pearce2010}). A hybrid finite volume and finite difference splitting approach was presented for numerical simulation of the fully nonlinear and weakly dispersive Green-Naghdi model (\cite{Bonneton2011,Bonneton2011-1,Chazel2011,Lannes2015}). In particular, numerical investigations on a dispersive-effect-improved Green-Naghdi model have been reported in \cite{Chazel2011,Lannes2015}.   A well-balanced  central discontinuous Galerkin (CDG) method coupling with the finite element (FE) method (\cite{LiM2014}) was employed to solve the 1D fully nonlinear weakly dispersive Green-Naghdi model over varying topography. Recent works of other kinds of discontinuous Galerkin (DG) methods for the Green-Naghdi model include \cite{Duran2015,Panda2014}.

The DG method is a class of high order finite element methods, which was originally introduced in 1973 by Reed and Hill (\cite{Reed1973}) for the neutron transport equation. Thereafter, it has achieved great progress in a series of pioneer papers (\cite{Cock1989,Cock1989-1,Cock1990,Cock1998}). The DG method has its own advantages in dealing with numerical solutions for many problems in sciences and engineering, consisting of being very flexible to achieve high order of accuracy and handle complicated geometry and boundary conditions, and to name a few. As a variant of the DG method, the CDG method is also one of popular high order numerical methods defined on overlapping meshes, which was originally introduced for hyperbolic conservation laws (\cite{Liu2007}), and then for diffusion equations (\cite{Liu2008}). By evolving two sets of numerical solutions defined on overlapping meshes which provide more information on numerical solutions, the CDG method does not rely on any exact or approximate Riemann solver at element interfaces as in the simulation of the DG method. The CDG method  has  been successfully applied to solve various partial differential equations, such as the Hamilton-Jacobi equations (\cite{LiF2010}), the ideal MHD equations (\cite{LiF2011,LiF2012}), the Euler equations (\cite{LiM2016}) and  the shallow water equations (\cite{LiM2014,LiM2017}). Recently, a reconstructed CDG method has been developed in \cite{Dong2017} for the improvement of computational efficiency, and two kinds of CDG methods defined on unstructured overlapping meshes have also been presented in \cite{Xu2016,LiM2019} for the treatment of complex computational domains.

In this paper, we first derive a formulation describing the propagation of 2D fully nonlinear weakly dispersive water waves,  using the same approach as in \cite{Su1969} where the authors derived the 1D equation over the flat bottom. Based on this formulation, we derive a Green-Naghdi model with an enhanced dispersive effect, and then carry out a linear dispersion analysis for the corresponding linear equation to show the improvement on dispersive effects. This improvement on computational modeling is of great importance during the procedure of seeking more accurate numerical solutions of shallow water waves. In simulation of the derived 2D Green-Naghdi model, we usually encounter several difficulties:  dealing with the mixed spatial and temporal derivative terms in the flux gradient and the source term, particularly in the dispersive-effect-improved model to be considered in the current work;  maintaining the nonnegativity of the water depth and preserving the still-water stationary solution;  computational efficiency and accuracy in 2D simulation. Concerning the first challenge,  we reformulate the model into a hybrid system of the pseudo-conservative form with a nonhomogeneous source term  and the pseudo-elliptic equation (\cite{MGH10,LiM2014}).  Since the resulting dynamic equations are no longer hyperbolic,  we utilize the CDG method (\cite{Liu2007}) as the base scheme which has an important feature of being free of Riemann solvers. Meanwhile, we employ the FE method  to solve the elliptic part, and point out that there is no need for the CDG-FE hybrid method to make special treatments on the continuity and discontinuity due to the fact of two copies of solutions on overlapping meshes. Secondly, the reason for naming the pseudo-elliptic equation in our numerical model lies in the fact that the computed variable of water depth appears in the coefficients of the second-order derivative term and may be trivial in the dry area. On the other hand, it explains that the nonnegativity of computational water depth is extremely important in our numerical model and numerical scheme, and this issue has not been solved in the previous 1D work (\cite{LiM2014}) yet. In this work, we design a CDG method for the Green-Naghdi model not only maintaining the well-balanced property for the stationary solution but also preserving the nonnegativity for the solution at the dry area.   In addition to these issues on numerical modeling and schemes, computational efficiency (\cite{Xu2009,Lannes2015}) is also significantly important for 2D water wave simulations.  Since the CDG method is a local scheme similar to the DG method, the computation is implemented element by element at each time step and is thus friendly to the parallel execution. The computation load of solving the FE equation is dominant in our computation, and its fast solver will be considered in the future work.

The remainder of the paper is organized as follows. In section 2, we present a fully nonlinear strongly dispersive water wave model in 2D domain. In section 3, we propose a family of high order schemes, coupling positivity-preserving well-balanced CDG methods and continuous FE methods. In section 4, we perform a series of numerical experiments to demonstrate the well-balanced property, positivity-preserving property, high order accuracy as well as the capability of the Green-Naghdi model to describe the propagation of strongly nonlinear and dispersive waves. Some concluding remarks are given in section 5.

\section{Mathematical models}

\subsection{Green-Naghdi model}
\label{sec:2.1}

In this paper, we consider the following 2D  Green-Naghdi model
\begin{equation}\label{GN-NFB-2D}
\left\{\begin{array}{lclcl}
h_t+(hu)_x+(hv)_y=0,\\
(hu)_t+\left( hu^2+ \frac{1}{2}gh^2+\frac{1}{3}h^3 \Phi + \frac{1}{2} h^2 \Psi \right)_x + \left( huv \right)_y =-\left(gh+\frac{1}{2}h^2\Phi+h\Psi\right)b_x,\\
(hv)_t+\left(huv\right)_x+\left( hv^2+ \frac{1}{2}gh^2+\frac{1}{3}h^3\Phi+\frac{1}{2}h^2\Psi\right)_y=-\left(gh+\frac{1}{2}h^2\Phi+h\Psi\right)b_y,
\end{array} \right.
\end{equation}
where
\begin{eqnarray}
\Phi &=& - u_{xt}-u u_{xx}+ u_x^2 - v_{yt}-v v_{yy}+ v_y^2 - uv_{xy}- u_{xy}v+ 2 u_x v_y, \label{phi2D} \\
\Psi &=& b_x u_{t}+b_x u  u_{x}+b_{xx} u^2 + b_y v_{t} + b_y v v_y +b_{yy} v^2 + b_y u v_x + b_x u_y v + 2b_{xy}uv. \label{psi2D}
\end{eqnarray}
The derivation of the model is shown in Appendix.

\subsection{Green-Naghdi model with enhanced dispersion effect}
\label{sec:2.2}
We can observe from the last two equations in \eqref{GN-NFB-2D} that (\cite{Chazel2011})
\begin{eqnarray*}
u_t &=& -g(h+b)_x-uu_x-u_yv+\text{higher order terms}\\
    &\simeq& \alpha u_t+(1-\alpha)(-g(h+b)_x-uu_x-u_yv),
\end{eqnarray*}
and
\begin{eqnarray*}
v_t &=& -g(h+b)_y-vv_y-uv_x+\text{higher order terms}\\
    &\simeq& \alpha v_t+(1-\alpha)(-g(h+b)_y-vv_y-uv_x).
\end{eqnarray*}
Replacing $u_t$ and $v_t$ with $\alpha u_t+(\alpha-1)(uu_x+u_yv+g(h+b)_x)$ and $\alpha v_t+(\alpha-1)(vv_y+uv_x+g(h+b)_y)$ in \eqref{phi2D} and \eqref{psi2D}, respectively,  we obtain a modified  Green-Naghdi model as follows
\begin{equation}\label{GN-NFB-2Da}
\left\{\begin{array}{lclcl}
h_t+(hu)_x+(hv)_y=0,\\
(hu)_t+\left( hu^2+ \frac{1}{2}gh^2+\frac{1}{3}h^3 \Phi + \frac{1}{2} h^2 \Psi \right)_x + \left( huv \right)_y =-\left(gh+\frac{1}{2}h^2\Phi+h\Psi\right)b_x,\\
(hv)_t+\left(huv\right)_x+\left( hv^2+ \frac{1}{2}gh^2+\frac{1}{3}h^3\Phi+\frac{1}{2}h^2\Psi\right)_y=-\left(gh+\frac{1}{2}h^2\Phi+h\Psi\right)b_y.
\end{array} \right.
\end{equation}
with
\begin{eqnarray}
\Phi &=& - \alpha u_{tx}-(\alpha-2)u_x^2-\alpha uu_{xx}-\alpha u_{xy}v-2(\alpha-1)u_yv_x-(\alpha-1)g(h+b)_{xx} \nonumber \\
     & & -\alpha v_{ty}-(\alpha-2)v_y^2-\alpha vv_{yy}-\alpha uv_{xy}-(\alpha-1)g(h+b)_{yy}+2u_xv_y, \label{phi2Da}\\
\Psi &=& \alpha b_xu_t+\alpha b_xuu_x+\alpha b_xu_yv+(\alpha-1)gb_x(h+b)_x+ \alpha b_yv_t \nonumber\\
     & & +\alpha b_yvv_y+\alpha b_yuv_x+(\alpha-1)gb_y(h+b)_y+b_{xx}u^2+b_{yy}v^2+2b_{xy}uv.  \label{psi2Da}
\end{eqnarray}

It is apparent that the standard Green-Naghdi model corresponds to a particular case of the modified Green-Naghdi model with $\alpha=1$.

\subsection{Linear dispersive analysis}
\label{sec:2.3}

We will carry out a linear dispersion analysis for the modified Green-Naghdi model given in the previous section. We first linearize the model \eqref{GN-NFB-2Da} for the flat-bottom case ($b=\text{constant}$) with the trivial solution $h = h_0 > 0$ ($h_0$ is a constant), $u = u_0 = 0$ and $v = v_0 = 0$, i.e. looking for solutions of the form

\begin{equation*}
h=h_0+\tilde{h}, \ \ u=\tilde{u}, \ \ v=\tilde{v},
\end{equation*}
where $(\tilde{h},\tilde{u},\tilde{v})$ are small perturbations, and retain only linear terms. As a result,  we have
\begin{equation*}
\left\{ {\begin{array}{*{20}{l}}
\tilde{h}_t + h_0\tilde{u}_x + h_0\tilde{v}_y = 0~,\\
\tilde{u}_t + g\tilde{h}_x - \frac{1}{3}h_0^2 \left(\alpha \tilde{u}_{txx} + \alpha \tilde{v}_{tyx} + (\alpha-1)\tilde{h}_{xxx} + (\alpha-1)\tilde{h}_{yyx}\right) = 0~, \\
\tilde{v}_t + g\tilde{h}_y - \frac{1}{3}h_0^2 \left(\alpha \tilde{u}_{txy} + \alpha \tilde{v}_{tyy} + (\alpha-1)\tilde{h}_{xxy} + (\alpha-1)\tilde{h}_{yyy}\right) = 0~.
\end{array}} \right.
\end{equation*}
Then we look for perturbations as plane waves taking the form
\begin{equation*}
\left( \begin{array}{l}
\tilde{h}\\
\tilde{u}\\
\tilde{v}
\end{array} \right) = \left( \begin{array}{l}
\hat{h}\\
\hat{u}\\
\hat{v}
\end{array} \right){e^{i(k_1x+k_2y - \omega t)}},
\end{equation*}
where $k_1$ and $k_2$ are the wave numbers, and $\omega$ is the frequency of the perturbations. Without loss of generality, $k_1$ and $k_2$ are taken to be real and positive. We end up with the linear system
\begin{equation*}
\left( {\begin{array}{*{20}{c}}
 -\omega  &  h_0k_1  &  h_0k_2\\
gk_1 + \frac{\alpha-1}{3}gh_0^2k_1(k_1^2+k_2^2)  & -\omega(1 + \frac{\alpha}{3}h_0^2 k_1^2 ) & -\omega \frac{\alpha}{3}h_0^2 k_1 k_2  \\
gk_2 + \frac{\alpha-1}{3}gh_0^2k_2(k_1^2+k_2^2)  & -\omega \frac{\alpha}{3}h_0^2 k_1 k_2    & -\omega(1 + \frac{\alpha}{3}h_0^2 k_2^2 )
\end{array}} \right)\left( {\begin{array}{*{20}{c}}
\hat{h}\\
\hat{u}\\
\hat{v}
\end{array}} \right) = \left( {\begin{array}{*{20}{c}}
0\\
0\\
0
\end{array}} \right),
\end{equation*}
and it has nontrivial solutions provided that the determinant of the coefficient matrix is zero, i.e.
\begin{equation}\label{dispersion relation}
\omega  = {\omega _ \pm } =  \pm |k| \sqrt {g{h_0}(1 + \frac{{\alpha  - 1}}{3}h_0^2{|k|^2}){{(1 + \frac{\alpha }{3}h_0^2{|k|^2})}^{ - 1}}}
\end{equation}
with $|k|=\sqrt{k_1^2+k_2^2}$.
This gives the linear dispersion relation for the modified Green-Naghdi equations \eqref{GN-NFB-2Da} in the flat-bottom case with respect to the trivial solution.    In Figure \ref{Fig:Dispersion}, we compare the linear dispersion relation \eqref{dispersion relation} ($\alpha=1.0$  and $\alpha=1.159$) and that of the full water wave problem in finite depth,
\begin{equation*}
\omega  = {\omega _ \pm } =  \pm \sqrt {g |k| \tanh ({h_0}|k|)},
\end{equation*}
with $g = 1$, $h_0 = 1$.
It can be observed from this figure that compared with the case $\alpha=1.0$, the linear dispersion relation of the modified model with $\alpha=1.159$ shows a better agreement with that of the full water wave problem.

\begin{figure}[H]
\begin{center}
\includegraphics[height=8cm,width=12cm,angle=0]{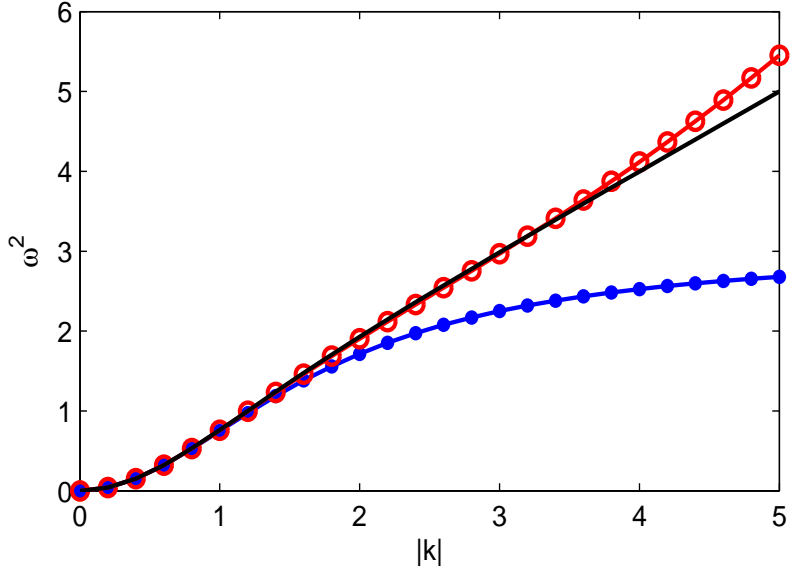}
\caption{Comparison between the linear dispersion relation \eqref{dispersion relation} of the Green-Naghdi equations with $\alpha=1.0$ (blue line with dots) and $\alpha=1.159$ (red line with circles) and the exact linear dispersion relation of the full water wave problem (black solid line), with $g = 1$, $h_0 = 1$.}
\label{Fig:Dispersion}
\end{center}
\end{figure}

\section{Numerical methods for the Green-Naghdi models}
\label{sec:3}
The first difficulty in designing numerical schemes for the Green-Naghdi model comes from the appearance of  mixed spatial and temporal derivatives in the equations. To tackle this, we make a reformulation from the original equations through introducing  auxiliary variables, and the procedure will be given in Section \ref{sec:3.1}. This new model will be adopted for the 2D shallow water wave simulations. In addition, in the case of a variable bottom, these equations admit still-water stationary solutions for the system \eqref{GN-NFB-2Da} which are given by
\begin{equation}\label{Eq:SSW2}
u=0~,\qquad v=0~,\qquad h+b = \mbox{constant}~.
\end{equation}
Numerical errors in discretization will produce  the spurious oscillations which lead to a wrong set of solutions. Another major difficulty in the simulation of shallow water waves has to do with the appearance of dry areas. In particular, this case happens frequently in so many important applications involving rapidly moving interfaces between wet and dry areas, for instance, the wave run-up on beaches or over man-made structures, dam break and tsunami. If no special care is taken, non-physical phenomena may arise and it will lead to a breakdown of computation. Following the work in \cite{Xing2010,LiM2017}, we design in Section \ref{sec:3.2} a high order positivity-preserving well-balanced CDG-FE method for the 2D simulation of the modified Green-Naghdi model in Section \ref{sec:3.1}.

\subsection{Reformulation of modified Green-Naghdi model}
\label{sec:3.1}
In order to remove the mixed spatial and temporal derivatives in the equations, we introduce two new unknown variables $P$ and $Q$ which satisfy:
\begin{eqnarray}\label{Eq:hP}
hP &=& -\left(\frac{\alpha}{3}h^3 u_x + \frac{\alpha}{3}h^3 v_y - \frac{\alpha}{2}h^2 v b_y \right)_x-\left(\frac{\alpha}{2}h^2 v b_x \right)_y \nonumber \\
   & & + h \Big(1+\alpha h_xb_x+\frac{\alpha}{2}h b_{xx} +\alpha b_x^2 \Big)u+h \Big(\alpha h_yb_x+\frac{\alpha}{2}h b_{xy} +\alpha b_x b_y \Big)v~,
\end{eqnarray}
and
\begin{eqnarray}\label{Eq:hQ}
hQ &=& -\left(\frac{\alpha}{3}h^3 u_x + \frac{\alpha}{3}h^3 v_y - \frac{\alpha}{2}h^2ub_x \right)_y-\left(\frac{\alpha}{2}h^2 u b_y \right)_x \nonumber \\
   & & + h \Big(\alpha h_xb_y+\frac{\alpha}{2}h b_{xy} +\alpha b_x b_y \Big)u + h \Big(1+\alpha h_yb_y+\frac{\alpha}{2}h b_{yy} +\alpha b_y^2 \Big)v~.
\end{eqnarray}
Then the system \eqref{GN-NFB-2Da} can be reformulated into a balance law
\begin{equation}\label{GN-NFB-2D-law}
\bU_t+\mathbf{F}(\bU,u,v;b)_x+\mathbf{G}(\bU,u,v;b)_y=\mathbf{S}(\bU,u,v;b),
\end{equation}
where $\mathbf{U}=(h,hP,hQ)^\top$ is the unknown vector,
\begin{eqnarray}\label{GN-NFB-2D-Flux1}
\mathbf{F}(\bU,u,v; b)&=&\left(hu, hPu +hQv +\frac{1}{2}gh^2 -\alpha huvb_xb_y+\frac{1-\alpha}{2}h^2(u^2b_{xx}+v^2b_{yy})\right.\nonumber \\
& & \left.-hv^2(1+\alpha b_y^2)-(\frac{4\alpha-2}{3}h^3u_x^2+\frac{6\alpha-2}{3}h^3u_xv_y+\frac{4\alpha-2}{3}h^3v_y^2)\right.\nonumber  \\
& & \left. +\alpha h^2u(u_x+v_y)b_x+\frac{3}{2}\alpha h^2v( u_x+v_y)b_y -\frac{2}{3}(\alpha-1)h^3u_yv_x +(1-\alpha)h^2uvb_{xy}\right. \nonumber \\
& & \left. -\frac{\alpha-1}{3}gh^3((h+b)_{xx}+(h+b)_{yy})+\frac{\alpha-1}{2}gh^2(b_x(h+b)_x+b_y(h+b)_y), \right.\nonumber \\
& & \left. huv(1+\alpha b_y^2)+\alpha hu^2b_xb_y-\frac{\alpha}{2}h^2u(u_x+v_y)b_y \right)^\top \nonumber \\
\end{eqnarray}
and
\begin{eqnarray}\label{GN-NFB-2D-Flux2}
\mathbf{G}(\bU,u,v; b)&=&\left(hv, huv(1+\alpha b_x^2)+\alpha hv^2b_xb_y-\frac{\alpha}{2}h^2v(u_x+v_y)b_x,  hPu +hQv   \right.  \nonumber \\
& & \left. +\frac{1}{2}gh^2 -\alpha huvb_xb_y+\frac{1-\alpha}{2}h^2(u^2b_{xx}+v^2b_{yy})-hu^2(1+\alpha b_x^2)\right.  \nonumber \\
& & \left. - (\frac{4\alpha-2}{3}h^3u_x^2+\frac{6\alpha-2}{3}h^3u_xv_y+\frac{4\alpha-2}{3}h^3v_y^2)+\frac{3}{2}\alpha h^2u (u_x +v_y)b_x    \right.  \nonumber \\
& & \left. + \alpha h^2v (u_x + v_y )b_y-\frac{2}{3}(\alpha-1)h^3u_yv_x +(1-\alpha)h^2uvb_{xy}\right.\nonumber \\
& & \left. -\frac{\alpha-1}{3}gh^3((h+b)_{xx}+(h+b)_{yy})+\frac{\alpha-1}{2}gh^2(b_x(h+b)_x+b_y(h+b)_y)\right)^\top \nonumber \\
\end{eqnarray}
are the flux terms, and
\begin{eqnarray}\label{GN-NFB-2D-source}
\bS(\bU,u,v; b)&=&\left(0, -ghb_x-\frac{\alpha}{2}h^2 u (u_x+v_y) b_{xx} -\frac{\alpha}{2}h^2 v (u_x+v_y) b_{xy}+(2\alpha-1)h u^2 b_x b_{xx}    \right.  \nonumber \\
& & \left.  +huv((3\alpha-2)b_xb_{xy}+\alpha b_{xx}b_y)+(\alpha-1)h^2(u_x^2+u_xv_y+u_yv_x+v_y^2)b_x \right.  \nonumber \\
& & \left. + \alpha hv^2b_{xy}b_y +(\alpha-1)hv^2b_xb_{yy} +\frac{\alpha-1}{2}gh^2((h+b)_{xx}+(h+b)_{yy})b_x  \right.   \nonumber \\
& & \left. -(\alpha-1)gh(b_x^2(h+b)_x+b_xb_y(h+b)_y), \right.   \nonumber \\
& & \left. -ghb_y-\frac{\alpha}{2}h^2 u (u_x+v_y) b_{xy}-\frac{\alpha}{2}h^2 v (u_x+v_y) b_{yy} +(2\alpha-1)h v^2 b_y b_{yy}\right.  \nonumber \\
& & \left. +huv((3\alpha-2)b_{xy}b_y+\alpha b_xb_{yy})+(\alpha-1)h^2(u_x^2+u_xv_y+u_yv_x+v_y^2)b_y \right.  \nonumber \\
& & \left. +\alpha hu^2b_xb_{xy}+(\alpha-1)hu^2b_{xx}b_y +\frac{\alpha-1}{2}gh^2((h+b)_{xx}+(h+b)_{yy})b_y \right.  \nonumber \\
& & \left. -(\alpha-1)gh(b_xb_y(h+b)_x+b_y^2(h+b)_y)           \right)^\top \nonumber \\
\end{eqnarray}
is the source term.

Based on the new equations, the solution of the modified Green-Naghdi model \eqref{GN-NFB-2Da} amounts to finding the unknowns ${(h, hP, hQ)^\top}$ based on \eqref{GN-NFB-2D-law} and the unknowns ${(u,v)^\top}$  from \eqref{Eq:hP}-\eqref{Eq:hQ}. We remark that the linear dispersive relationship of the reformulated Green-Naghdi model \eqref{Eq:hP}-\eqref{GN-NFB-2D-law} is the same as the one of \eqref{GN-NFB-2Da}.

\subsection{CDG-FE methods}
\label{sec:3.2}

In this section, we develop the numerical method for the solution of the equations \eqref{Eq:hP}-\eqref{GN-NFB-2D-law}. Let $\mathcal{T}^C=\{C_{ij}, \forall i, j \}$ and $\mathcal{T}^D=\{D_{ij}, \forall i, j \}$ define two overlapping meshes for the computational domain  $\Omega=[x_{\min}, x_{\max}] \times [y_{\min}, y_{\max}]$, with $C_{ij}=[x_{i-\frac{1}{2}}, x_{i+\frac{1}{2}}] \times [y_{j-\frac{1}{2}}, y_{j+\frac{1}{2}}]$, $D_{ij}=[x_{i-1}, x_{i}] \times [y_{j-1}, y_{j}]$, $x_{i+\frac{1}{2}}=\frac{1}{2}(x_{i}+x_{i+1})$ and $y_{j+\frac{1}{2}}=\frac{1}{2}(y_{j}+y_{j+1})$, where $\{x_i\}_i$ and $\{y_j\}_j$ are partitions of $[x_{\min}, x_{\max}]$ and $[y_{\min}, y_{\max}]$ respectively. Associated with each mesh, we define the following discrete spaces
\begin{eqnarray*}
\mathcal {V}^C&=&\mathcal {V}^{C,k}=\{{\bf{v}}:{\bf{v}}|_{C_{ij}}\in [P^k(C_{ij})]^3~, \forall\, i, j \}~,\\
\mathcal {V}^D&=&\mathcal {V}^{D,k}=\{{\bf{v}}:{\bf{v}}|_{D_{ij}}\in [P^k(D_{ij})]^3~, \,
\forall \, i, j \}~.
\end{eqnarray*}
To approximate $u$ and $v$, we define two continuous finite element spaces
\begin{eqnarray*}
\mathcal {W}^C&=&\mathcal{W}^{C,k}=\{w:w|_{C_{ij}}\in P^k({C_{ij}})~, \forall\, i,j \text{ and $w$ is continuous} \}~,\\
\mathcal {W}^D&=&\mathcal W^{D,k}=\{w:w|_{D_{ij}}\in P^k({D_{ij}})~, \,
\forall \,i,j \text{ and $w$ is continuous} \}~.
\end{eqnarray*}
We only present the schemes with the forward Euler method for time discretization. High-order time discretizations will be discussed in Section \ref{sec:3.3}. The proposed methods evolve two copies of numerical solution, which are assumed to be available at $t=t_n$, denoted by $\bU^{n,\star}=(h^{n,\star}, (hP)^{n,\star}, (hQ)^{n,\star})^\top \in \mathcal{V}^\star$, and we want to find the solutions at $t=t_{n+1}=t_n+\Delta t_n$.  Only the procedure to update $\bU^{n+1,C}$ will be described.  We project the bottom topography function $b$ into $P^k(C_{ij})$ on $C_{ij}$ (resp. into $P^k(D_{ij})$ on $D_{ij}$) in the $L^2$ sense, and obtain an approximation $b^{C}$ (resp. $b^{D}$) throughout the computational domain.

\subsubsection{Standard CDG-FE method}
\label{sec:3.2.1}

We first apply to \eqref{GN-NFB-2D-law} the standard CDG methods of Liu et al. (\cite{Liu2007}) for space discretization and the forward Euler method for time discretization. That is, we look for $\bU^{n+1,C}=(h^{n+1,C}, (hP)^{n+1,C}, (hQ)^{n+1,C})^\top\in\mathcal{V}^{C,k}$ such that for any $\bV\in\mathcal{V}^{C,k}|_{C_{ij}}$ with any $i$ and $j$,
\begin{eqnarray}\label{Comp-SW-2D-1}
\int_{C_{ij}} \bU^{n+1,C}\cdot \bV dxdy &=& \int_{C_{ij}}\left( \theta \bU^{n,D} + (1-\theta) \bU^{n,C}\right)\cdot \bV dxdy   \nonumber\\
&+&\Delta t_n   \int_{C_{ij}} \left[ \bF(\bU^{n,D},u^{n,D},v^{n,D};b^D)\cdot \bV_x + \bG(\bU^{n,D},u^{n,D},v^{n,D};b^D)\cdot \bV_y \right] dxdy  \nonumber\\
&-&\Delta t_n \int_{y_{j-\frac{1}{2}}}^{y_{j+\frac{1}{2}}}\left[\bF(\bU^{n,D}(x_{i+\frac{1}{2}},y),u^{n,D}(x_{i+\frac{1}{2}},y),
v^{n,D}(x_{i+\frac{1}{2}},y);b^D(x_{i+\frac{1}{2}},y)) \cdot \bV(x^-_{i+\frac{1}{2}},y)\right. \nonumber\\
&-&\left.\bF(\bU^{n,D}(x_{i-\frac{1}{2}},y),u^{n,D}(x_{i-\frac{1}{2}},y),v^{n,D}(x_{i-\frac{1}{2}},y);
b^D(x_{i-\frac{1}{2}},y)) \cdot \bV(x^+_{i-\frac{1}{2}},y) \right] dy \nonumber\\
&-&\Delta t_n \int_{x_{i-\frac{1}{2}}}^{x_{i+\frac{1}{2}}}\left[\bG(\bU^{n,D}(x,y_{j+\frac{1}{2}}),u^{n,D}(x,y_{j+\frac{1}{2}}),
v^{n,D}(x,y_{j+\frac{1}{2}});b^D(x,y_{j+\frac{1}{2}})) \cdot \bV(x,y^-_{j+\frac{1}{2}})\right. \nonumber\\
&-&\left.\bG(\bU^{n,D}(x,y_{j-\frac{1}{2}}),u^{n,D}(x,y_{j-\frac{1}{2}}),v^{n,D}(x,y_{j-\frac{1}{2}});
b^D(x,y_{j-\frac{1}{2}})) \cdot \bV(x,y^+_{j-\frac{1}{2}}) \right] dx \nonumber\\
&+& \Delta t_n\int_{C_{ij}}\bS(\bU^{n,D},u^{n,D},v^{n,D}; b^{D}) \cdot \bV dxdy~.
\end{eqnarray}
Here $\theta=\Delta t_n/\tau \in [0, 1]$ with $\tau$ being the maximal time step allowed by the CFL restriction (\cite{Liu2007}). In general, this numerical scheme does not necessarily maintain the still-water stationary solution \eqref{Eq:SSW2} and preserve the non-negativity of water depth.

Once $\mathbf{U}^{n+1,C}$ is available, we can obtain $u^{n+1,C}$ and $v^{n+1,C}$by applying a continuous finite element method to \eqref{Eq:hP} and \eqref{Eq:hQ}: look for $u^{n+1,C}, v^{n+1,C} \in \widetilde{W}^{C,k}$ such that for any $\widehat{u}, \widehat{v} \in \widehat{W}^{C,k}$,
\begin{eqnarray}
\label{FE1}
& &\int_\Omega \left(\frac{\alpha}{3}(h^{n+1,C})^3(u^{n+1,C})_x+\frac{\alpha}{3}(h^{n+1,C})^3(v^{n+1,C})_y
-\frac{\alpha}{2}(h^{n+1,C})^2v^{n+1,C}(b^C)_y \right)\widehat{u}_xdxdy \nonumber\\
&+&\int_\Omega \left(\frac{\alpha}{2}(h^{n+1,C})^2v^{n+1,C}(b^C)_x \right)\widehat{u}_ydxdy
+\int_\Omega f(h^{n+1,C},u^{n+1,C},v^{n+1,C};b^C)\widehat{u}dxdy \nonumber\\
&=&\int_\Omega (hP)^{n+1,C}\widehat{u}dxdy,
\end{eqnarray}
\begin{eqnarray}
\label{FE2}
& &\int_\Omega \left(\frac{\alpha}{3}(h^{n+1,C})^3(u^{n+1,C})_x+\frac{\alpha}{3}(h^{n+1,C})^3(v^{n+1,C})_y
-\frac{\alpha}{2}(h^{n+1,C})^2u^{n+1,C}(b^C)_x \right)\widehat{v}_ydxdy \nonumber\\
&+&\int_\Omega \left(\frac{\alpha}{2}(h^{n+1,C})^2u^{n+1,C}(b^C)_y \right)\widehat{v}_xdxdy
+\int_\Omega g(h^{n+1,C},u^{n+1,C},v^{n+1,C};b^C)\widehat{v}dxdy \nonumber\\
&=&\int_\Omega (hQ)^{n+1,C}\widehat{v}dxdy,
\end{eqnarray}
where
\begin{eqnarray}
f(h^{n+1,C},u^{n+1,C},v^{n+1,C};b^C)&=&h^{n+1,C}\left (1+\alpha(h^{n+1,C})_x(b^C)_x+\frac{\alpha}{2}h^{n+1,C}(b^C)_{xx}+\alpha(b^C)_{x}^2 \right)u^{n+1,C} \nonumber\\
&+&h^{n+1,C}\left (\alpha(h^{n+1,C})_y(b^C)_x+\frac{\alpha}{2}h^{n+1,C}(b^C)_{xy}+\alpha (b^C)_x (b^C)_y \right)v^{n+1,C}, \nonumber \\
\end{eqnarray}
\begin{eqnarray}
g(h^{n+1,C},u^{n+1,C},v^{n+1,C};b^C)&=&h^{n+1,C}\left (\alpha(h^{n+1,C})_x(b^C)_y+\frac{\alpha}{2}h^{n+1,C}(b^C)_{xy}+\alpha (b^C)_x (b^C)_y \right)u^{n+1,C}, \nonumber \\
&+&h^{n+1,C}\left (1+\alpha(h^{n+1,C})_y(b^C)_y+\frac{\alpha}{2}h^{n+1,C}(b^C)_{yy}+\alpha(b^C)_{y}^2 \right)v^{n+1,C}, \nonumber \\
\end{eqnarray}
and $\widetilde{W}^{C,k}$, $\widehat{W}^{C,k}$ are variants of $\mathcal{W}^{C,k}$ with consideration of the boundary conditions (\cite{LiM2014}).

\begin{remark}
When the bottom is flat ($b=\mbox{constant}$), the unique solvability of this FE method can be obtained in a straightforward manner if $h \geq h_0 >0$ and $\alpha \geq \alpha_0>0$.  When the bottom is not flat, the unique solvability of this FE scheme is more difficult to be determined. Even though $h \geq h_0 >0$ and $\alpha \geq \alpha_0>0$, a constraint condition for the $b$ as in the 1D case in \cite{LiM2014} is needed for the unique solvability of the FE method. Due to the complexity of the 2D case, the condition can not be obtained explicitly. However, except for the tests in Sections 4.2(Case B) and 4.3 which involve a dry area or a near dry area, the FE system of the other numerical applications in Section 4 (satisfy $h \geq h_0 >0$ and $\alpha \geq \alpha_0>0$) was found to be uniquely solvable in numerics.
\end{remark}

\begin{remark}
For the cases in Sections 4.2(Case B) and 4.3 which involve a dry area or a near dry area, the FE equation is ill-conditioned, so we shall evaluate the velocity $u$ using an approximation technique. We solve the FE equation on the cells with $h^{n+1,C} \geq h_0 >0$, while for other cells, we omit the high order terms in \eqref{Eq:hP}-\eqref{Eq:hQ} and get
\begin{eqnarray}\label{Eq:hP-1}
hP & \simeq &  h \Big(1 +\alpha b_x^2 \Big)u+\alpha h b_x b_y v~,\\
hQ & \simeq &  \alpha h b_x b_y u + h \Big(1+\alpha b_y^2 \Big)v~.
\end{eqnarray}
Then the velocity $u$ and $v$ can be evaluated by
\begin{eqnarray}\label{Eq:hP-2}
u & \simeq &  \frac{1+\alpha b_y^2}{1 +\alpha b_x^2 + \alpha b_y^2} \cdot \frac{hP}{h} - \frac{\alpha b_x b_y}{1 +\alpha b_x^2 + \alpha b_y^2} \cdot \frac{hQ}{h},\\ \label{Eq:hP-3}
v & \simeq &  \frac{1+\alpha b_x^2}{1 +\alpha b_x^2 + \alpha b_y^2} \cdot \frac{hQ}{h} - \frac{\alpha b_x b_y}{1 +\alpha b_x^2 + \alpha b_y^2} \cdot \frac{hP}{h},
\end{eqnarray}
To deal with the singularity in \eqref{Eq:hP-2}-\eqref{Eq:hP-3}, we employ the regularization technique presented in \cite{Kurganov2007} and thus we have:
\begin{eqnarray}\label{Eq:hP-4}
u & \simeq &  \frac{1+\alpha b_y^2}{1 +\alpha b_x^2 + \alpha b_y^2} \cdot \widetilde{P} - \frac{\alpha b_x b_y}{1 +\alpha b_x^2 + \alpha b_y^2} \cdot \widetilde{Q} ,\\ \label{Eq:hP-5}
v & \simeq &  \frac{1+\alpha b_x^2}{1 +\alpha b_x^2 + \alpha b_y^2} \cdot \widetilde{Q}  - \frac{\alpha b_x b_y}{1 +\alpha b_x^2 + \alpha b_y^2} \cdot \widetilde{P} ,
\end{eqnarray}
where
\begin{eqnarray}\label{Eq:hP-6}
\widetilde{P}=\frac{\sqrt{2}h (hP)}{\sqrt{h^4+\max(h^4,\varepsilon)}}, \widetilde{Q}=\frac{\sqrt{2}h (hQ)}{\sqrt{h^4+\max(h^4,\varepsilon)}}
\end{eqnarray}
with $\varepsilon=\min((\Delta x)^{4},(\Delta y)^{4})$. In the computation, $h_0=\max((\Delta x)^{k+1},(\Delta y)^{k+1})$ for $k=1,2$.
\end{remark}

\subsubsection{Well-balanced CDG-FE method}
\label{sec:3.2.2}

In this subsection, we propose a family of high order well-balanced CDG-FE schemes for the model \eqref{Eq:hP}-\eqref{GN-NFB-2D-law}, which exactly preserves the still-water steady state solution \eqref{Eq:SSW2}. The well-balance property can be achieved  by adding some terms to the scheme \eqref{Comp-SW-2D-1},
\begin{eqnarray}\label{Comp-WB-SW-2D-0}
& &\int_{C_{ij}} \bU_h^{n+1,C}\cdot \bV dxdy = \mbox{right-hand side of \eqref{Comp-SW-2D-1}}  + \btU(b_h^{C}, b_h^{D},\bV)\nonumber\\
&+& \Delta t_n \int_{y_{j-\frac{1}{2}}}^{y_{j+\frac{1}{2}}} \left[\btS_1(\bU_h^{n,D}; b_h^{D}(x_i^+,y))-\btS_1(\bU_h^{n,D}; b_h^{D}(x_i^-,y))\right] \cdot \bV(x_i,y) dy   \nonumber\\
&+&\Delta t_n \int_{x_{i-\frac{1}{2}}}^{x_{i+\frac{1}{2}}} \left[\btS_2(\bU_h^{n,D}; b_h^{D}(x,y_j^+))-\btS_2(\bU_h^{n,D}; b_h^{D}(x,y_j^-))\right] \cdot \bV(x,y_j) dx~,\nonumber
\end{eqnarray}
where the correction terms are given by
\begin{equation}
\label{Solution-Correction-1}
\btU(b_h^{C}, b_h^{D}, \bV)= \theta\int_{C_{ij}}\left( b_h^{D}-b_h^D, 0, 0 \right)^\top\cdot \bV dxdy~,
\end{equation}
\begin{equation}
\label{Source-Correction-1}
\btS_1(\bU_h^{n,D}; b_h^{D})=\left(0, \frac{g}{2}(b_h^D)^2-\gamma^{n,D}_{ij} g b_h^D, 0 \right)^\top~,
\end{equation}
\begin{equation}
\label{Source-Correction-2}
\btS_2(\bU_h^{n,D}; b_h^{D})=\left(0, 0, \frac{g}{2}(b_h^D)^2-\gamma^{n,D}_{ij} g b_h^D\right)^\top~.
\end{equation}
Here, $\gamma^{n,D}_{ij}$ is a special constant which represents the average value of the water surface $\eta_h^{n,D}=h_h^{n,D}+b_h^{D}$ in the element  $C_{ij}$. Particularly, we can take
\[
\begin{array}{lclcl}
  \gamma^{n,D}_{ij} &=&  \frac{1}{4}\left[\eta_h^{n,D}(x_{i+\frac{1}{2}},y_{j+\frac{1}{2}}) +\eta_h^{n,D}(x_{i-\frac{1}{2}},y_{j+\frac{1}{2}})\right. \\
  &+&\left. \eta_h^{n,D}(x_{i+\frac{1}{2}},y_{j-\frac{1}{2}}) + \eta_h^{n,D}(x_{i-\frac{1}{2}}, y_{j-\frac{1}{2}})\right]~.
\end{array}
\]
With the following decomposition of the source term
\begin{eqnarray}\label{decom-2D}
\bS(\bU_h^{n,D},u_h^{n,D},v_h^{n,D}; b_h^{D})&=& \left(0,-g(\eta_h^{n,D}-b_h^D)(b_h^D)_x, -g(\eta_h^{n,D}-b_h^D)(b_h^D)_y \right)^{\top} + \mbox{the remaining items} \nonumber\\
&=& \left(0,\frac{g}{2}(b_h^D)^2-g\gamma_{ij}^{n,D}b_h^D, 0 \right)_x^{\top}  + \left(0, 0, \frac{g}{2}(b_h^D)^2-g\gamma_{ij}^{n,D}b_h^D \right)_y^{\top} \nonumber\\
&-& \left(0, g(\eta_h^{n,D}-\gamma_{ij}^{n,D})(b_h^D)_x, g(\eta_h^{n,D}-\gamma_{ij}^{n,D})(b_h^D)_y \right)^{\top} \nonumber\\
&+& \mbox{the remaining items},
\end{eqnarray}
the scheme \eqref{Comp-WB-SW-2D-0} can be rewritten as
\begin{eqnarray}\label{Comp-WB-SW-2D-1}
& &\int_{C_{ij}} \bU_h^{n+1,C}\cdot \bV dxdy \nonumber\\
&=& \int_{C_{ij}}\left( \theta \bU_h^{n,D} + (1-\theta) \bU_h^{n,C}\right)\cdot \bV dxdy + \btU(b_h^{C}, b_h^{D}, \bV)\nonumber\\
 &+&\Delta t_n   \int_{C_{ij}} \left[ \bF(\bU_h^{n,D},u_h^{n,D},v_h^{n,D};b_h^D)\cdot \bV_x + \bG(\bU_h^{n,D},u_h^{n,D},v_h^{n,D};b_h^D)\cdot \bV_y \right] dxdy  \nonumber\\
&-&\Delta t_n \int_{y_{j-\frac{1}{2}}}^{y_{j+\frac{1}{2}}}\left[\bF(\bU_h^{n,D}(x_{i+\frac{1}{2}},y),u_h^{n,D}(x_{i+\frac{1}{2}},y),
v_h^{n,D}(x_{i+\frac{1}{2}},y);b_h^D(x_{i+\frac{1}{2}},y)) \cdot \bV(x^-_{i+\frac{1}{2}},y)\right. \nonumber\\
&-&\left.\bF(\bU_h^{n,D}(x_{i-\frac{1}{2}},y),u_h^{n,D}(x_{i-\frac{1}{2}},y),v_h^{n,D}(x_{i-\frac{1}{2}},y);
b_h^D(x_{i-\frac{1}{2}},y)) \cdot \bV(x^+_{i-\frac{1}{2}},y) \right] dy \nonumber\\
&-&\Delta t_n \int_{x_{i-\frac{1}{2}}}^{x_{i+\frac{1}{2}}}\left[\bG(\bU_h^{n,D}(x,y_{j+\frac{1}{2}}),u_h^{n,D}(x,y_{j+\frac{1}{2}}),
v_h^{n,D}(x,y_{j+\frac{1}{2}});b_h^D(x,y_{j+\frac{1}{2}})) \cdot \bV(x,y^-_{j+\frac{1}{2}})\right. \nonumber\\
&-&\left.\bG(\bU_h^{n,D}(x,y_{j-\frac{1}{2}}),u_h^{n,D}(x,y_{j-\frac{1}{2}}),v_h^{n,D}(x,y_{j-\frac{1}{2}});
b_h^D(x,y_{j-\frac{1}{2}})) \cdot \bV(x,y^+_{j-\frac{1}{2}}) \right] dx \nonumber\\
&-& \Delta t_n  \int_{C_{ij}}\left(0, g(\eta_h^{n,D}-\gamma_{ij}^{n,D})(b_h^D)_x, g(\eta_h^{n,D}-\gamma_{ij}^{n,D})(b_h^D)_y \right)^{\top} \cdot \bV dxdy  \nonumber\\
&-& \Delta t_n  \int_{C_{ij}}\left(0, \frac{g}{2}(b_h^D)^2-g\gamma_{ij}^{n,D})b_h^D,0 \right)^{\top} \cdot \bV_x dxdy  \nonumber\\
&-& \Delta t_n  \int_{C_{ij}}\left(0, 0, \frac{g}{2}(b_h^D)^2-g\gamma_{ij}^{n,D})b_h^D \right)^{\top} \cdot \bV_y dxdy  \nonumber\\
&+& \Delta t_n  \int_{y_{j-\frac{1}{2}}}^{y_{j+\frac{1}{2}}} \left(0, \frac{g}{2}(b_h^D(x_{i+\frac{1}{2}},y))^2-g\gamma_{ij}^{n,D})b_h^D(x_{i+\frac{1}{2}},y) ,0 \right)^{\top} \cdot \bV(x_{i+\frac{1}{2}}^-,y) dy \nonumber\\
&-& \Delta t_n  \int_{y_{j-\frac{1}{2}}}^{y_{j+\frac{1}{2}}} \left(0, \frac{g}{2}(b_h^D(x_{i-\frac{1}{2}},y))^2-g\gamma_{ij}^{n,D})b_h^D(x_{i-\frac{1}{2}},y) ,0 \right)^{\top} \cdot \bV(x_{i-\frac{1}{2}}^+,y) dy \nonumber\\
&+& \Delta t_n \int_{x_{i-\frac{1}{2}}}^{x_{i+\frac{1}{2}}} \left(0 ,0, \frac{g}{2}(b_h^D(x,y_{j+\frac{1}{2}}))^2-g\gamma_{ij}^{n,D})b_h^D(x,y_{j+\frac{1}{2}}) \right)^{\top} \cdot \bV(x,y_{j+\frac{1}{2}}^-) dx \nonumber\\
&-& \Delta t_n  \int_{x_{i-\frac{1}{2}}}^{x_{i+\frac{1}{2}}} \left(0,0,\frac{g}{2}(b_h^D(x,y_{j-\frac{1}{2}}))^2-g\gamma_{ij}^{n,D})b_h^D(x,y_{j-\frac{1}{2}}) \right)^{\top} \cdot \bV(x,y_{j-\frac{1}{2}}^+) dx \nonumber\\
&+& \mbox{the remaining items}~.
\end{eqnarray}

\begin{prop}
\label{Th3.1}
The numerical scheme, defined in \eqref{Comp-WB-SW-2D-1}, \eqref{FE1} and \eqref{FE2} and their counterparts for  $\bU_h^{n+1,D}$, $u_h^{n+1,D}$ and $v_h^{n+1,D}$ ,  to solve the 2D Green-Naghdi model \eqref{Eq:hP}-\eqref{GN-NFB-2D-law} is well-balanced, in the sense that it preserves the still-water stationary solution (\ref{Eq:SSW2}).
\end{prop}

\begin{proof}
The proof can be easily obtained by the mathematical induction with time step $n$ as the Proposition in \cite{LiM2014}.
\end{proof}

\subsubsection{Positivity-preserving CDG-FE method}
\label{sec:3.2.3}
We now discuss the positivity-preserving CDG-FE method  for  \eqref{Eq:hP}-\eqref{GN-NFB-2D-law}. Firstly, let $\hat L_i^{1,x}=\{ \hat x_i^{1,\beta}, \beta=1,2,...,{\hat N} \}$ and $\hat L_i^{2,x}=\{\hat x_i^{2, \beta}, \beta=1,2,...,{\hat N}\}$ be the Legendre Gauss-Lobatto quadrature points on $[x_{i-\frac{1}{2}}, x_{i}]$ and $[x_{i}, x_{i+\frac{1}{2}}]$ respectively, while $\hat L_j^{1,y}=\{ \hat y_j^{1,\beta}, \beta=1,2,...,{\hat N} \}$ and $\hat L_j^{2,y}=\{\hat y_j^{2, \beta}, \beta=1,2,...,{\hat N}\}$ represent the Legendre Gauss-Lobatto quadrature points on $[y_{j-\frac{1}{2}}, y_{j}]$ and $[y_{j}, y_{j+\frac{1}{2}}]$ respectively, $\forall i, j$. The corresponding quadrature weights on the reference element $[-\frac{1}{2}, \frac{1}{2}]$ are $\hat \omega_{\beta}, \beta =1, 2, ... , {\hat N}$, and $\hat N$ is chosen such that $2 \hat N-3 \geq k$. In addition, let $L_i^{1,x}=\{ x_i^{1,\alpha}, \alpha=1,2,...,N \}$ and $L_i^{2,x}=\{x_i^{2, \alpha}, \alpha=1,2,...,N\}$ denote the Gaussian quadrature points on $[x_{i-\frac{1}{2}}, x_{i}]$ and $[x_{i}, x_{i+\frac{1}{2}}]$ respectively, while $L_j^{1,y}=\{ y_j^{1,\alpha}, \alpha=1,2,...,N \}$ and $L_j^{2,y}=\{y_j^{2, \alpha}, \alpha=1,2,...,N\}$ represent the Gaussian quadrature points on $[y_{j-\frac{1}{2}}, y_{j}]$ and $[y_{j}, y_{j+\frac{1}{2}}]$ respectively. The corresponding quadrature weights $\omega_{\alpha}, \alpha=1, 2, ... , N$ are distributed on the interval $[-\frac{1}{2}, \frac{1}{2}]$ and $N$ is chosen such that the Gaussian quadrature is exact for the integration of univariate polynomials of degree $2k+1$. Define $L_{i,j}^{l,m}=(L_i^{l,x}\otimes \hat L_j^{m,y})\cup (\hat L_i^{l,x}\otimes L_j^{m,y})$ with $l,m=1,2$. By taking the test function $\bV=(\frac{1}{\Delta x\Delta y}, 0, 0)^\top$, we get the equation satisfied by the cell average of $h^{n,C}$,
\begin{eqnarray}\label{gene-cell-2D}
\bar h_{ij}^{n+1,C}&=& (1-\theta) \bar h_{ij}^{n,C} + \frac{\theta}{\Delta x \Delta y}  \int_{C_{ij}}h^{n,D}dxdy   \nonumber\\
&-&\frac{\Delta t_n}{\Delta x \Delta y} \int_{y_{j-\frac{1}{2}}}^{y_{j+\frac{1}{2}}}\left[h^{n,D}(x_{i+\frac{1}{2}},y)u^{n,D}(x_{i+\frac{1}{2}},y)-
h^{n,D}(x_{i-\frac{1}{2}},y)u^{n,D}(x_{i-\frac{1}{2}},y)\right] dy \nonumber\\
&-&\frac{\Delta t_n}{\Delta x \Delta y} \int_{x_{i-\frac{1}{2}}}^{x_{i+\frac{1}{2}}}\left[h^{n,D}(x,y_{j+\frac{1}{2}})v^{n,D}(x,y_{j+\frac{1}{2}})-
h^{n,D}(x,y_{j-\frac{1}{2}})v^{n,D}(x,y_{j-\frac{1}{2}})\right] dx~,
\end{eqnarray}
where $\bar h_{ij}^{n,C}$ denotes the cell average of the CDG solution $h^C$ on $C_{ij}$ at time $t_n$.

In the numerical implementation, the definite integrals in the intervals $[x_{i-\frac{1}{2}},x_{i+\frac{1}{2}}]$ and $[y_{j-\frac{1}{2}},y_{j+\frac{1}{2}}]$ are evaluated by applying the Gaussian quadrature rule described above to each half of the interval (also see Section 3.2 in \cite{Cheng2012}), the scheme \eqref{gene-cell-2D} becomes
\begin{eqnarray}\label{cell1-2D-1}
\bar{h}_{ij}^{n+1,C}
&=& (1-\theta) \bar{h}_{ij}^{n,C} + \frac{\theta}{\Delta x\Delta y}\int_{C_{ij}} h^{n,D} dxdy\nonumber\\
&-&\frac{\Delta t_n}{2\Delta x}\sum_{m=1}^{2}\sum_{\alpha=1}^{N}\omega_{\alpha}\left[ h^{n,D}\left(\hat x_i^{2,\hat N},y_j^{m,\alpha}\right)u^{n,D}\left(\hat x_i^{2,\hat N},y_j^{m,\alpha}\right)\right. \nonumber\\
&-&\left.h^{n,D}\left(\hat x_i^{1,1},y_j^{m,\alpha}\right)u^{n,D}\left(\hat x_i^{1,1},y_j^{m,\alpha}\right) \right] \nonumber\\
&-&\frac{\Delta t_n}{2\Delta y}\sum_{l=1}^{2}\sum_{\alpha=1}^{N}\omega_{\alpha}\left[ h^{n,D}\left(x_i^{l,\alpha}, \hat y_j^{2,\hat N}\right)v^{n,D}\left(x_i^{l,\alpha}, \hat y_j^{2,\hat N}\right)\right. \nonumber\\
&-&\left.h^{n,D}\left(x_i^{l,\alpha},\hat y_j^{1,1}\right)v^{n,D}\left(x_i^{l,\alpha}, \hat y_j^{1,1}\right) \right],
\end{eqnarray}
herein, we used $\hat x_i^{1,1}=x_{i-\frac{1}{2}}, \hat x_i^{2,\hat N}=x_{i+\frac{1}{2}}, \hat y_j^{1,1}=y_{j-\frac{1}{2}}, \hat y_j^{2,\hat N}=y_{j+\frac{1}{2}}$. Now we have the following result.

\begin{prop}
\label{Th3.2}
For any given $n\ge 0$, we assume $\bar h_{ij}^{n,C}\ge 0$ and $\bar h_{ij}^{n,D}\ge 0$, $\forall i, j$.
Consider the scheme in \eqref{cell1-2D-1} and its counterpart for $\bar h_{ij}^{n+1,D}$, if $h^C(x,y,t_n)\ge 0$ and $h^D(x,y,t_n)\ge 0$, $\forall (x,y) \in L_{i,j}^{l,m}, \forall i,j$ with $l,m=1,2$, then $\bar h_{ij}^{n+1, C} \ge 0$ and $\bar h_{ij}^{n+1, D} \ge 0$, $\forall i,j$, under the CFL condition
\begin{equation}\label{gene-CFL-con-2d}
\lambda_x a_x +\lambda_y a_y \leq \frac{1}{4} \theta  \hat \omega_1~,
\end{equation}
where  $\lambda_x=\Delta t_n/\Delta x$, $\lambda_y=\Delta t_n/\Delta y$, $a_x=\max(\|u^{n,C}\|_\infty, \|u^{n,D}\|_\infty)$, $a_y=\max(\|v_h^{n,C}\|_\infty, \|v_h^{n,D}\|_\infty)$.
\end{prop}

\begin{proof}
Since the numerical solution $h^{n,D}$ is a piecewise polynomial with degree $k$, the integral of $h^{n,D}$ in cell $C_{ij}$ in \eqref{cell1-2D-1} is exactly evaluated in our numerical implementation.  However, in order to discuss the non-negativity of the cell average $\bar{h}_{ij}^{n+1,C}$. We equivalently evaluate the integral using a combination of the Gauss quadrature rule and the Legendre Gauss-Lobatto quadrature rule as follows:
\begin{eqnarray}\label{cell1-2D-2}
\frac{\theta}{\Delta x\Delta y}\int_{C_{ij}} h^{n,D} dxdy &=& \frac{\theta}{2\Delta x\Delta y}\int_{C_{ij}} h^{n,D} dxdy + \frac{\theta}{2\Delta x\Delta y}\int_{C_{ij}} h^{n,D} dxdy \nonumber\\
&=& \frac{\theta}{8}\sum_{l,m=1}^{2}\sum_{\beta=1}^{\hat N} \sum_{\alpha=1}^{N} \hat {\omega}_{\beta} \omega_{\alpha} h^{n,D}\left(\hat x_i^{l,\beta},y_j^{m,\alpha}\right)  + \frac{\theta}{8}\sum_{l,m=1}^{2}\sum_{\alpha=1}^{N} \sum_{\beta=1}^{\hat N} \omega_{\alpha} \hat{\omega}_{\beta} h^{n,D}\left(x_i^{l,\alpha}, \hat y_j^{m,\beta}\right) \nonumber\\
\end{eqnarray}
Plugging \eqref{cell1-2D-2} into \eqref{cell1-2D-1}, one obtains
\begin{eqnarray}\label{cell1-2D-3}
\bar{h}_{ij}^{n+1,C}
&=& (1-\theta) \bar{h}_{ij}^{n,C} \nonumber\\
&+&\sum_{m=1}^{2}\sum_{\alpha=1}^{N}\omega_{\alpha}\left[ \frac{\theta}{8} \hat {\omega}_{1}-\frac{\Delta t_n}{2\Delta x} u^{n,D}\left(\hat x_i^{2,\hat N},y_j^{m,\alpha}\right) \right]h^{n,D}\left(\hat x_i^{2,\hat N},y_j^{m,\alpha}\right)  \nonumber\\
&+&\sum_{m=1}^{2}\sum_{\alpha=1}^{N}\omega_{\alpha}\left[ \frac{\theta}{8} \hat {\omega}_{1} + \frac{\Delta t_n}{2\Delta x} u^{n,D}\left(\hat x_i^{1,1},y_j^{m,\alpha}\right) \right]h^{n,D}\left(\hat x_i^{1,1},y_j^{m,\alpha}\right)  \nonumber\\
&+&\sum_{l=1}^{2}\sum_{\alpha=1}^{N}\omega_{\alpha}\left[ \frac{\theta}{8} \hat {\omega}_{1}-\frac{\Delta t_n}{2\Delta y} v^{n,D}\left(x_i^{l,\alpha}, \hat y_j^{2,\hat N}\right) \right] h^{n,D}\left(x_i^{l,\alpha}, \hat y_j^{2,\hat N}\right)  \nonumber\\
&+&\sum_{l=1}^{2}\sum_{\alpha=1}^{N}\omega_{\alpha}\left[ \frac{\theta}{8} \hat {\omega}_{1} + \frac{\Delta t_n}{2\Delta y} v^{n,D}\left(x_i^{l,\alpha},\hat y_j^{1,1}\right) \right]h^{n,D}\left(x_i^{l,\alpha}, \hat y_j^{1,1}\right)   \nonumber\\
&+&\frac{\theta}{8}\sum_{l,m=1}^{2}\sum_{\beta=2}^{\hat N - 1} \sum_{\alpha=1}^{N} \hat {\omega}_{\beta} \omega_{\alpha} h^{n,D}\left(\hat x_i^{l,\beta},y_j^{m,\alpha}\right) \nonumber\\
&+&\frac{\theta}{8}\sum_{m=1}^{2} \sum_{\alpha=1}^{N} \omega_{\alpha} \hat {\omega}_{1} \left( h^{n,D}\left(\hat x_i^{2,1},y_j^{m,\alpha}\right) +  h^{n,D}\left(\hat x_i^{1,\hat N},y_j^{m,\alpha}\right) \right) \nonumber\\
&+&\frac{\theta}{8}\sum_{l,m=1}^{2}\sum_{\alpha=1}^{N} \sum_{\beta=2}^{\hat N - 1} \omega_{\alpha} \hat{\omega}_{\beta} h^{n,D}\left(x_i^{l,\alpha}, \hat y_j^{m,\beta}\right) \nonumber\\
&+&\frac{\theta}{8}\sum_{l=1}^{2}\sum_{\alpha=1}^{N}  \omega_{\alpha} \hat{\omega}_{1} \left( h^{n,D}\left(x_i^{l,\alpha}, \hat y_j^{2,1}\right) +  h^{n,D}\left(x_i^{l,\alpha}, \hat y_j^{1,\hat N}\right) \right).
\end{eqnarray}
Here, we used $\hat{\omega}_{\hat N}=\hat{\omega}_{1}$.
A few observations can be made. Firstly, $\bar{h}_{ij}^{n+1,C}$ is a linear combination of $\bar{h}_{ij}^{n,C}$, $h^{n,D}\left(x_i^{l,\alpha}, \hat y_j^{m,\beta}\right)$ and $h^{n,D}\left(\hat x_i^{l,\beta},y_j^{m,\alpha}\right)$, $l,m=1,2,\alpha=1,2,...,N,\beta==1,2,...,\hat N$, which are all non-negative according to the conditions in this Proposition. Secondly, the CFL condition \eqref{gene-CFL-con-2d} and $\theta \in [0,1]$ imply that all coefficients in the linear combination are non-negative. Therefore  $\bar{h}_{ij}^{n+1,C} \geq 0 $, $\forall i,j$. Similarly, one can show $\bar{h}_{ij}^{n+1,D} \geq 0 $, $\forall i,j$.
\end{proof}

\begin{remark}
Although the line to prove Proposition \ref{Th3.2} is similar to the one in \cite{LiM2017}, there is an explicit difference in the proof. In \cite{LiM2017}, the water depth $h$ and the momentum $hu$ may be not consistent due to the numerical error, namely, $hu\neq0$ if $h \equiv 0$. This issue hampers the proof of the non-negativity of the water depth. To overcome the issue, a numerical technique has been employed to modify the momentum to be consistent with the water depth. While according to the numerical scheme in the present paper, the velocity is always equal to zero as the water depth is equal to zero. Therefore, the proof for Proposition \ref{Th3.2} is more natural and simpler than the one in \cite{LiM2017}.
\end{remark}

Next, we give the positivity-preserving limiters which modify the CDG solution polynomials $h^{n,C}$ and $h^{n,D}$ into $\tilde h^{n,C}$ and $\tilde h^{n,D}$ which satisfy the sufficient condition given in Proposition \ref{Th3.2}. In fact, the limiters are the same as in \cite{Zhang2010,Zhang2010-1,Xing2010}, as long as the notation $K$ and $\hat L_K$ are re-defined as follows: On the primal mesh, $K$ denotes a mesh element $C_{ij}$ and $\hat L_K$ represents the set of relevant quadrature points in $K$, namely $\hat L_K=\cup_{l,m=1}^2 L_{i,j}^{l,m}$. On the dual mesh, $K$ denotes a mesh element $D_{ij}$ and $\hat L_K$ represents the set of relevant quadrature points in $K$, namely $\hat L_K=L_{i,j}^{1,1} \cup L_{i,j-1}^{1,2}\cup L_{i-1,j}^{2,1}\cup L_{i-1,j-1}^{2,2}$. Following \cite{Xing2010}, the positivity-preserving limiter is given as follows:
On each mesh element $K$, we modify the water depth $h^{n,\star} $ into $\tilde h^{n,\star}=\alpha_K(h^{n,\star}-\bar h^{n,\star})+ \bar h^{n,\star}$, with $\alpha_K=\min_{x \in \hat L_K} \left\{1, |\bar h^{n,\star} /(\bar h^{n,\star}- h^{n,\star}(x)) | \right\}$ and $\star=C,D$.

\subsubsection{Positivity-preserving well-balanced CDG-FE method}
\label{sec:3.2.4}

Finally, we study the positivity-preserving and well-balanced CDG-FE method for the 2D Green-Naghdi model \eqref{Eq:hP}-\eqref{GN-NFB-2D-law}. We start with the well-balanced CDG method satisfied by the cell average of the numerical solution $h^C$, which is obtained by taking the test function $\bV=(\frac{1}{\Delta x \Delta y}, 0, 0)^\top$ in the scheme (\ref{Comp-WB-SW-2D-1}),
\begin{eqnarray}\label{WB-gene-cell-2D}
\bar h_{ij}^{n+1,C}&=& (1-\theta) \bar h_{ij}^{n,C} + \frac{\theta}{\Delta x \Delta y}  \int_{C_{ij}}h^{n,D}dxdy   \nonumber\\
&-&\frac{\Delta t_n}{\Delta x \Delta y} \int_{y_{j-\frac{1}{2}}}^{y_{j+\frac{1}{2}}}\left[h^{n,D}(x_{i+\frac{1}{2}},y)u^{n,D}(x_{i+\frac{1}{2}},y)-
h^{n,D}(x_{i-\frac{1}{2}},y)u^{n,D}(x_{i-\frac{1}{2}},y)\right] dy \nonumber\\
&-&\frac{\Delta t_n}{\Delta x \Delta y} \int_{x_{i-\frac{1}{2}}}^{x_{i+\frac{1}{2}}}\left[h^{n,D}(x,y_{j+\frac{1}{2}})v^{n,D}(x,y_{j+\frac{1}{2}})-
h^{n,D}(x,y_{j-\frac{1}{2}})v^{n,D}(x,y_{j-\frac{1}{2}})\right] dx \nonumber\\
&+& \theta\left( \frac{1}{\Delta x \Delta y}\int_{C_{ij}}b^{D}\cdot dxdy - \bar b_{ij}^{C} \right)~,
\end{eqnarray}
where $\bar b_{ij}^C$ (resp. $\bar b_{ij}^D$)  is the cell average of the bottom topography $b^C$ (resp. $b^D$) on the element $C_{ij}$ (resp. $D_{ij}$).

\begin{prop}
\label{Th3.3}
For any given $n\ge 0$, we assume $\bar h_{ij}^{n,C}\ge 0$ and $\bar h_{ij}^{n,D}\ge 0$, $\forall i, j$. Consider the scheme in \eqref{WB-gene-cell-2D}  and its counterpart for $\bar h_{ij}^{n+1,D}$,   if
\begin{equation}\label{bot-con-2d}
\bar b_{ij}^C=\frac{1}{\Delta x \Delta y} \int_{C_{ij}} b^{D} dxdy~, \qquad \bar b_{ij}^D= \frac{1}{\Delta x \Delta y} \int_{D_{ij}} b^{C} dxdy~, \quad \forall i, j~,
\end{equation}
and $h^C(x,y,t_n)\ge 0$, $h^D(x,y,t_n)\ge 0$, $\forall (x,y) \in L_{i,j}^{l,m}$, $\forall i, j$ with $l,m=1,2$, then $\bar h_{ij}^{n+1, C}\ge 0$ and $\bar h_{ij}^{n+1, D} \ge 0$, $\forall i, j$, under the CFL condition
\begin{equation}\label{CFL_con_2D}
\lambda_x a_x + \lambda_y a_y \leq  \frac{1}{4}\theta \hat \omega_1~.
\end{equation}
\end{prop}

\begin{proof}
The proof is a direct result of Proposition \ref{Th3.2} and condition \eqref{bot-con-2d}.
\end{proof}

\begin{remark}
To enforce the sufficient conditions in Proposition \ref{Th3.3}, we also need to modify the approximations to the bottom and use the positivity-preserving limiter in Section \ref{sec:3.2.3}.
The new approximations to the bottom topography $b(x,y)$, still denoted by $b^C$ and $b^D$, are obtained by solving a constrained minimization problem similar to the one in \cite{LiM2017} with the Lagrange multiplier method.
\end{remark}

\subsection{High-order time discretizations and nonlinear limiters}
\label{sec:3.3}
To achieve better accuracy in time, the strong stability preserving (SSP) high-order time discretizations (\cite{Gottlieb2001}) will be used in the numerical simulations . Such discretizations can be written as a convex combination of the forward Euler method, and therefore the resulting SSP schemes are also well-balanced and positivity-preserving. In this paper, we use the third order TVD Runge-Kutta method for the time discretization.

When the CDG method is applied to nonlinear problems, nonlinear limiters are often needed to prevent numerical instabilities. In this work,  we use the total variation bounded (TVB) minmod slope limiter with parameter $M=10$ (\cite{Cock1998}) in a componentwise way as it is needed. This limiter is applied to $(h+b, hP, hQ)^\top$ and it is used prior to application of the positivity-preserving limiter.

\section{Numerical examples}

\subsection{Accuracy test}
\label{sec:4.1}
In this example, we test the convergence rate of the proposed CDG-FE method by varying the mesh size. The Green-Naghdi model \eqref{GN-NFB-2D} with $b=0$ has an exact solution given by (\cite{Su1969})
\begin{equation}\label{Eq:so}
\left\{\begin{array}{lcl}
  h(x,y,t)=h_1+(h_2-h_1)\mbox{sech}^2\left(\frac{x-Dt}{2}\sqrt{\frac{3(h_2-h_1)}{h_2h_1^2}} \right)\\
  u(x,y,t)=D\left(1-\frac{h_1}{h(x,t)} \right)~,
  v(x,y,t)=0,
\end{array}\right .
\end{equation}
where $h_1$ is the typical water depth, $h_2$ corresponds to the solitary wave crest and $D=\sqrt{gh_2}$ is the wave speed.
In this test, we employ a solitary wave with $h_1=1$ and $h_2=2.25$ in \eqref{Eq:so} which is initially located at $x=0$ and propagating in the positive $x$-direction.
The computational domain is $[-30, 50]\times[-1,1]$ and the final time is $1$. An outgoing boundary condition is used in the $x$-direction and a periodic boundary condition is used in the $y$-direction. We use regular meshes with $\Delta x=\Delta y=1,0.5,0.25,0.125$. The time step is $\Delta t =0.1 \Delta x$. We present $L^2 $ errors and orders of accuracy for $h$ and $u$ in Table \ref{table:smooth1:1}. The results show that the CDG-FE method is $(k+1)$st order accurate for $P^k$ with $k=1,2$ and therefore it is optimal with respect to the approximation properties of the discrete spaces.

\begin{table}[!ht]
\caption{$L^2$ errors and orders of accuracy of $(h,u)$.}
\centering
\begin{tabular}{lllllllllllll}
  \hline\label{accuracytest}
   &&$h$&&& &&&$u$&&& &\\
  \cline{3-7}\cline{9-13}
   $\Delta x$ &&$P^1$    &   && $P^2$ &   && $P^1$ &  && $P^2$ & \\
   \cline{3-4}\cline{6-7}\cline{9-10}\cline{12-13}
   &&$L^2$ error & Order && $L^2$ error & Order && $L^2$ error & Order&& $L^2$ error & Order\\
   \cline{1-13}
     1     && 2.28E-01   & ---   && 7.80E-02     & ---   &&  5.16E-01    &  ---   &&  1.05E-01    &  --- \\
     0.5   && 6.00E-02   & 1.93  && 9.94E-03     & 2.97  &&  1.27E-02    &  2.03  &&  1.47E-02    &  2.84 \\
     0.25  && 1.53E-02   & 1.97  && 1.27E-03     & 2.97  &&  2.94E-02    &  2.11  &&  1.84E-03    &  2.99 \\
     0.125 && 3.53E-03   & 2.12  && 1.64E-04     & 2.96  &&  7.03E-03    &  2.07  &&  2.29E-04    &  3.01\\
  \hline
\end{tabular}
\label{table:smooth1:1}
\end{table}

\subsection{Stationary solution}
\label{sec:4.2}
In this test, we validate the well-balanced feature and the positivity-preserving property of the proposed method as applied to continuous and discontinuous variable bottoms. The initial conditions are
\begin{equation}\label{SS-Initial2D}
u(x,y,0)=0~, \qquad  v(x,y,0)=0~, \qquad h(x,y,0)+b(x,y)=0.50001~,
\end{equation}
and the continuous bottom profile (Case A) is defined by
\begin{equation}
b(x,y)=\left\{\begin{array}{lclclcl}
0.2~,  & r\leq 0.3~,\\
0.5-r~,       & 0.3\leq r\leq 0.5~,\\
0~, & \mbox{otherwise}~,
\end{array}\right.
\end{equation}
with $r=\sqrt{x^2+y^2}$, while the discontinuous bottom profile (Case B) is given by
\begin{equation}
b(x,y)=\left\{\begin{array}{lcl}
0.5~, & -0.5\leq x,y\leq 0.5~,\\
0~, & \mbox{otherwise}~.
\end{array}\right .
\end{equation}
We choose $[-1,1]\times[-1,1]$ as the computational domain, divided into $20\times20$ elements, and use outgoing boundary conditions. We compute the solution up to $t=10$ by the well-balanced CDG methods. Notice that there exists a near dry area for the second case.

For these cases, the standard CDG-FE method usually fails to preserve the still-water stationary solution exactly. Especially for the second case, the standard CDG-FE method will produce negative water depth due to the numerical oscillation, and thus the computation will blow down. To demonstrate that the positivity-preserving well-balanced CDG-FE scheme indeed preserves the still-water stationary solution exactly (i.e., up to machine precision), we perform the computation in both single and double precision. The corresponding $L^2$ errors on the water surface $h+b$ and velocity $(u,v)$ are listed in Table \ref{table:stationary_solution2D} for both topographies. We see that the errors have orders of a magnitude consistent with the machine single and double precision, and thus the numerical results verify the well-balanced property and the positivity-preserving property.

\begin{table}[!ht]
\caption{ $L^{\infty}$ errors on $(h+b,u,v)$ for the stationary solution at $t=10$.}
\centering
\begin{tabular}{llllllllll}
&&&&&&\\
\hline
  Case       &&  precision   && $h+b$             && $u$            && $v$   \\
 \cline{1-1}     \cline{3-3}    \cline{5-5}          \cline{7-7}       \cline{9-9}
  A          &&  single      && 5.96E-08	      && 8.86E-10	 	&& 9.73E-10\\
  A          &&  double      && 2.56E-16	  	  && 2.81E-16	    && 2.23E-16\\
  B          &&  single      && 5.96E-08	  	  && 4.21E-09	    && 2.45E-09\\
  B          &&  double      && 4.93E-16	 	  && 5.32E-16	    && 2.33E-16\\
\hline
\end{tabular}
\label{table:stationary_solution2D}
\end{table}

\subsection{Solitary wave overtopping a seawall}
\label{sec:4.3}
In this test, we consider the simulation of a solitary wave overtopping a seawall, which has been studied experimentally in \cite{Hsiao2010} and numerically in \cite{Lannes2015}. This example is also used to investigate the validity of our positivity-preserving well-balanced scheme. The initial solitary wave ($h_1=0.2,h_2=0.27$ in \eqref{Eq:so}) and the bottom topography including the seawall are shown in Figure \ref{Fig:seawall} (along $y=0$). The computational domain is $[-5,20]\times[-0.2,0.2]$ discretized into $500 \times 8$ uniform elements.
The outgoing boundary condition is used in the $x$-direction and the periodic boundary condition is used in the $y$-direction.  For the standard CDG-FE method, negative water
depth was generated during the simulation at the dry or near dry areas, this causes inaccurate velocity which is used as the boundary condition when solving the elliptic equations, and then the numerical solution blows up within a few dozen time steps. For the positivity-preserving well-balanced CDG-FE method, however, the water depth remains non-negative during the entire simulation. The numerical surface profiles (along $y=0$) at time $t=5, 7.5, 12.5, 20$ are plotted in Figure \ref{Fig:seawall-profiles}. The time series of the wave elevation at several positions of gauges ($x=5.9, 7.6, 9.644, 10.462, 10.732$ and $11.12$, $y=0$) are illustrated in Figure \ref{Fig:seawall-gages}. The time origin has been shifted in order to compare with results reported in \cite{Hsiao2010}. Both models with $\alpha=1$ and $\alpha=1.159$ give almost the same results. Our numerical results are similar to experimental data and the numerical results in \cite{Hsiao2010}.

\begin{figure}
\begin{center}
\includegraphics[height=6.0cm,width=12.0cm,angle=0]{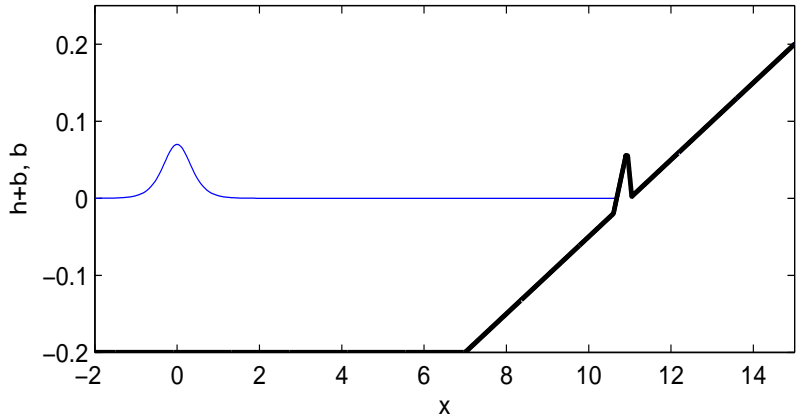}
\caption{The sketch of the topography and the initial wave for solitary wave overtopping a seawall.}
\label{Fig:seawall}
\end{center}
\end{figure}

\begin{figure}
\begin{center}
\includegraphics[height=4.5cm,width=8cm,angle=0]{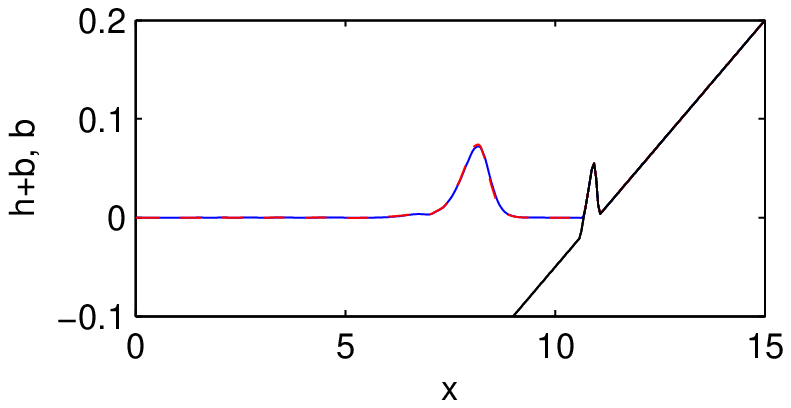}
\includegraphics[height=4.5cm,width=8cm,angle=0]{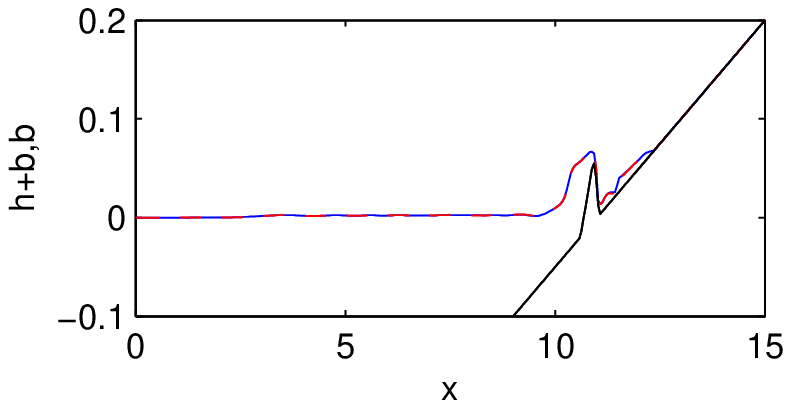}\\
\includegraphics[height=4.5cm,width=8cm,angle=0]{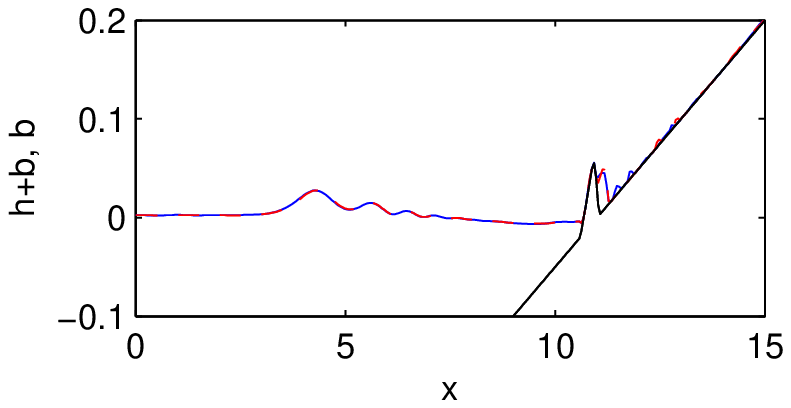}
\includegraphics[height=4.5cm,width=8cm,angle=0]{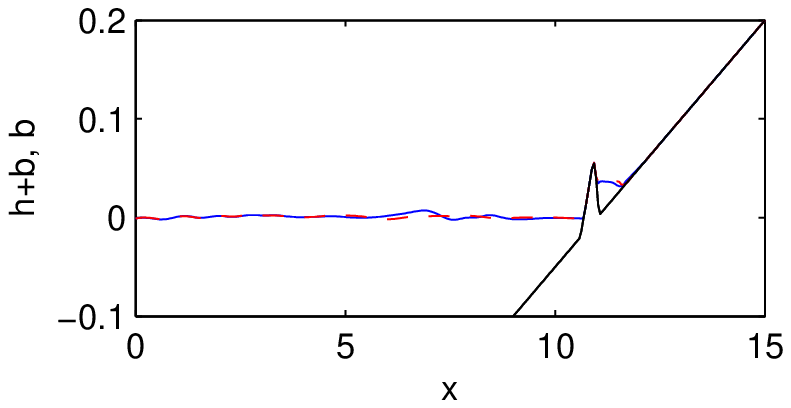}
\caption{Water surface at several times ($t=5, 7.5, 12.5, 20$, from left to right, from top to bottom). Blue line: numerical results with $\alpha=1$; Red dashed line: numerical results with $\alpha=1.159$.}
\label{Fig:seawall-profiles}
\end{center}
\end{figure}

\begin{figure}
\begin{center}
\includegraphics[height=2.6cm,width=8.0cm,angle=0]{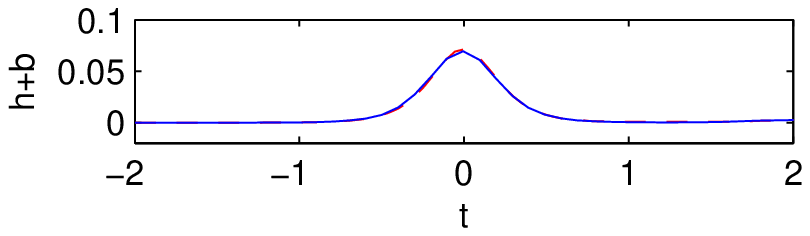}
\includegraphics[height=2.6cm,width=8.0cm,angle=0]{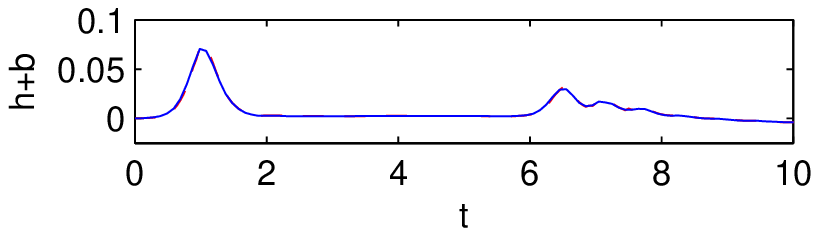}\\
\includegraphics[height=2.6cm,width=8.0cm,angle=0]{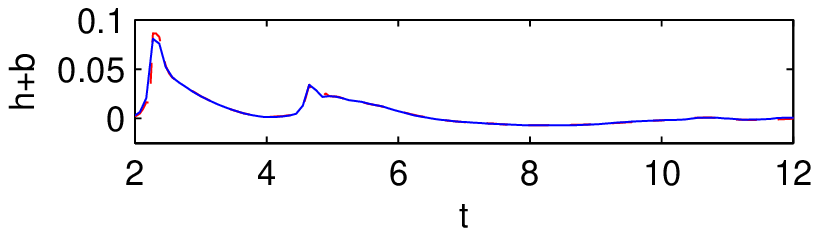}
\includegraphics[height=2.6cm,width=8.0cm,angle=0]{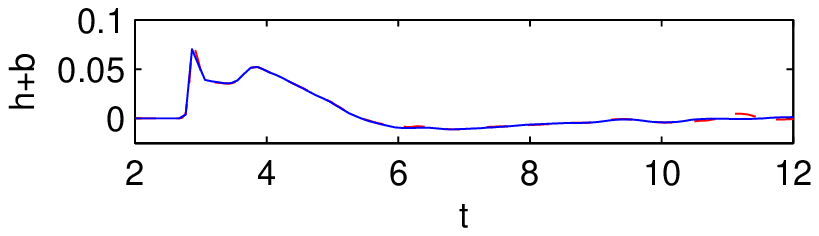}\\
\includegraphics[height=2.6cm,width=8.0cm,angle=0]{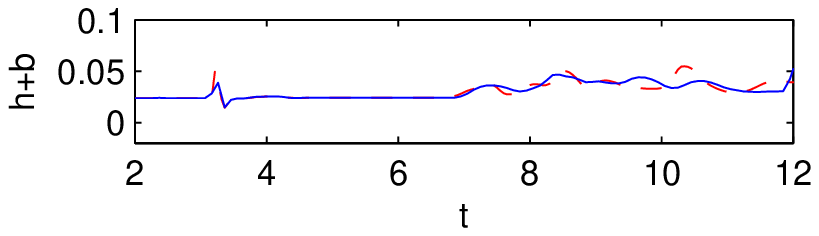}
\includegraphics[height=2.6cm,width=8.0cm,angle=0]{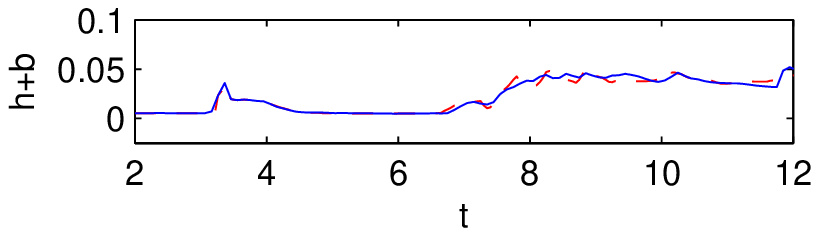}
\caption{The time series of the free surface elevation due to waves interacting against seawall at several gages ($x=5.9, 7.6, 9.644, 10.462, 10.732, 11.12$, from left to right, from top to bottom). Blue line: numerical results with $\alpha=1$; Red dashed line: numerical results with $\alpha=1.159$.}
\label{Fig:seawall-gages}
\end{center}
\end{figure}

\subsection{Periodic waves propagation over a submerged bar}
\label{sec:4.4}
In this example, we investigate the robustness of the modified Green-Naghdi model. We first consider the propagation of periodic Stokes waves over a submerged bar with plane slopes. The bottom variation is specified by
\begin{equation}\label{Eq:HarmonicB}
b(x,y)=\left\{\begin{array}{lcl}
-0.4+0.05(x-6), & 6\le x\le 12 \\
-0.1,        &  12 \le x \le 14\\
-0.1 - 0.1(x-14), & 14 \le x \le 17\\
-0.4,  & \mbox{elsewhere}~,
\end{array}\right.
\end{equation}
and is also exhibited in Figure \ref{Fig:harmonics_bottom} (along $y=0$) in which we also label the positions of $10$ gauges used in \cite{Ding1994}.

As shown in experimental work (\cite{Ding1994}), regular waves break up into higher-frequency free waves as they propagate past a submerged bar. As the waves travel up the front slope of the bar, higher harmonics are generated due to nonlinear interactions, causing the waves to steepen. These harmonics are then released as free waves on the downslope, producing an irregular pattern behind the bar. This experiment is particularly difficult to simulate because it includes nonlinear interactions and requires accurate propagation of waves in both deep and shallow water over a wide range of depths. Therefore it has often been used as a discriminating test case for nonlinear models of surface wave propagation over variable bottom (\cite{Chazel2011,Ding1994,Guyenne2007}).

In the simulation, the computational domain is $[0, 25]\times[-0.2, 0.2]$, divided into $500 \times 8$ uniform cells. At initial time, $h+b=0$ and $u=v=0$ in the computational domain.
The incident wave (entering from the left) is a third-order Stokes wave (\cite{Fenton1985}) given by
\begin{eqnarray}
\label{harmonics-incident}
\eta(x,t)&=&a_0 \cos\left( 2 \pi \left( \frac{x}{\lambda} - \frac{t}{T_0} \right)  \right)
+ \frac{\pi a_0^2}{\lambda} \cos\left( 4 \pi \left( \frac{x}{\lambda} - \frac{t}{T_0} \right)  \right) \nonumber \\
&-& \frac{\pi^2 a_0^3}{2 \lambda^2} \left[ \cos\left( 2 \pi \left( \frac{x}{\lambda} - \frac{t}{T_0} \right)  \right) - \cos\left( 6 \pi \left( \frac{x}{\lambda} - \frac{t}{T_0} \right)  \right) \right]~,
\end{eqnarray}
where $T_0$, $a_0$ and $\lambda$ denote the wave period, amplitude and wavelength, respectively. We choose $(T_0, a_0, \lambda) = (2.02, 0.01, 3.73)$ corresponding to one of the experiments in \cite{Ding1994}. An absorbing boundary condition is applied at the right boundary and a periodic boundary condition is used at the upper and bottom boundaries.

Figure \ref{Fig:harmonics6} depicts the time histories of the water surface at the first $6$ gauges ($x=2$, $4$, $10.5$, $12.5$, $13.5$ and $14.5$, $y=0$) and Figure \ref{Fig:harmonics4} depicts the time histories of the water surface at the last $4$ gauges ($x=15.7$, $17.3$, $19$ and $21$, $y=0$). The time origin has been shifted in order that the numerical results match the measurements for the first gauge at $x = 2, y=0$. We compare three sets of data:  experimental data (\cite{Ding1994}), numerical solutions from the original Green-Naghdi model ($\alpha=1$),  numerical solutions from the improved Green-Naghdi model ($\alpha=1.159$). It can be seen from  Figure \ref{Fig:harmonics6} that before the crest of the bar ($x \le 14.5$), the numerical results from both models ($\alpha=1$ and $\alpha=1.159$) match well with each other and compare well with experimental data.  However, it can be seen from Figure \ref{Fig:harmonics4} that, compared with the experimental data, discrepancies in amplitude and phase can be observed for gauges beyond the crest of the bar ($x \ge 15.7$) for the numerical solutions with $\alpha=1$, these discrepancies should be attributed to the weakly dispersive character of the original Green-Naghdi model, while for the modified Green-Naghdi model, the numerical solutions match well with the experimental data.

\begin{figure}
\begin{center}
\includegraphics[height=6.0cm,width=12.0cm,angle=0]{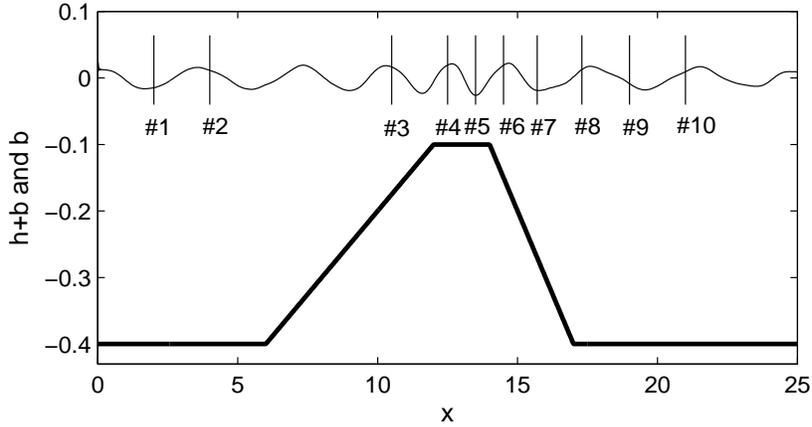}
\caption{Experimental set-up and locations of the wave gauges as used in \cite{Ding1994}.}
\label{Fig:harmonics_bottom}
\end{center}
\end{figure}

\begin{figure}
\begin{center}
\includegraphics[height=3.5cm,width=8.0cm,angle=0]{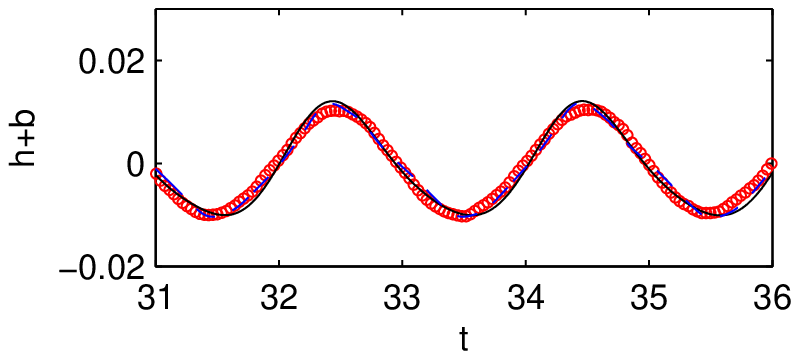}
\includegraphics[height=3.5cm,width=8.0cm,angle=0]{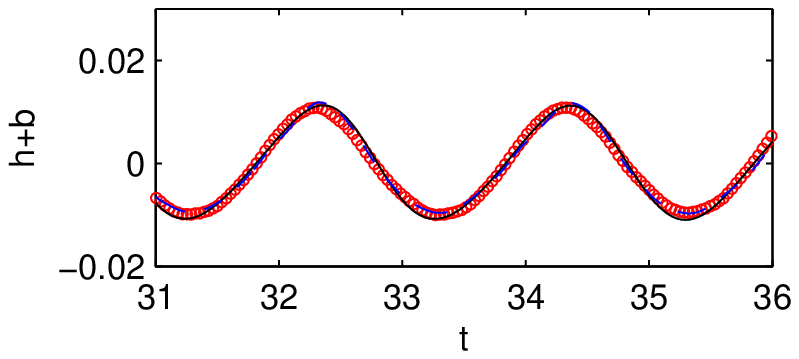}\\
\includegraphics[height=3.5cm,width=8.0cm,angle=0]{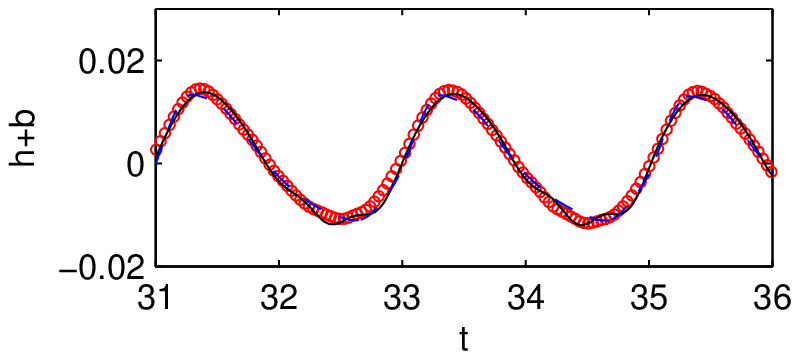}
\includegraphics[height=3.5cm,width=8.0cm,angle=0]{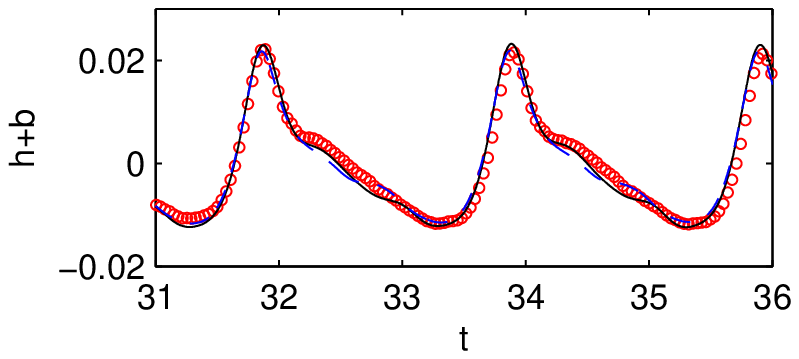}\\
\includegraphics[height=3.5cm,width=8.0cm,angle=0]{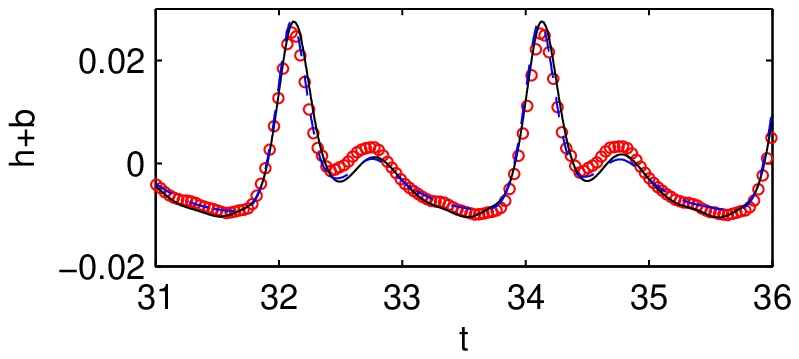}
\includegraphics[height=3.5cm,width=8.0cm,angle=0]{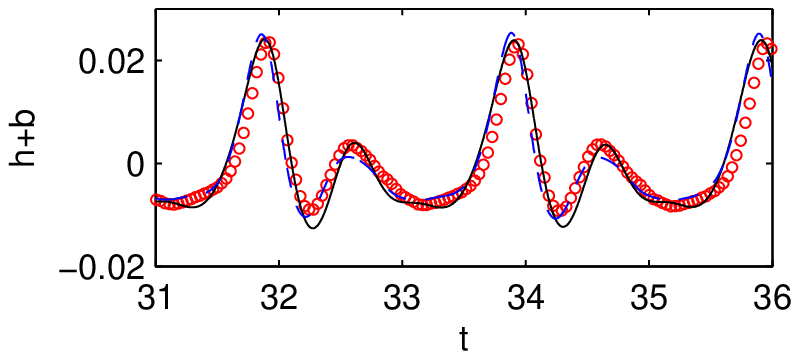}
\caption{Time series of surface elevations for waves passing over a submerged bar at $x=2$, $4$, $10.5$, $12.5$, $13.5$ and $14.5$ (from left to right, from top to bottom).  Circles: experimental data (\cite{Ding1994}), green solid line: numerical results with $\alpha=1$, blue solid line: numerical results with $\alpha=1.159$.}
\label{Fig:harmonics6}
\end{center}
\end{figure}

\begin{figure}
\begin{center}
\includegraphics[height=3.5cm,width=8.0cm,angle=0]{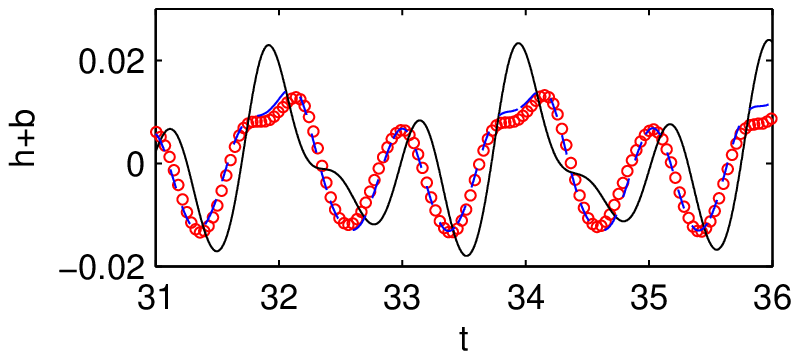}
\includegraphics[height=3.5cm,width=8.0cm,angle=0]{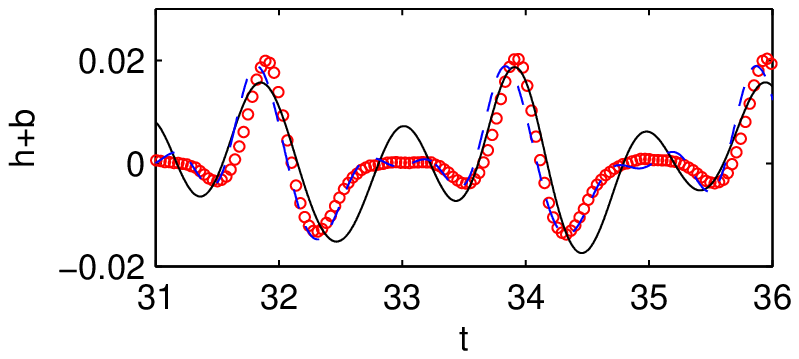}\\
\includegraphics[height=3.5cm,width=8.0cm,angle=0]{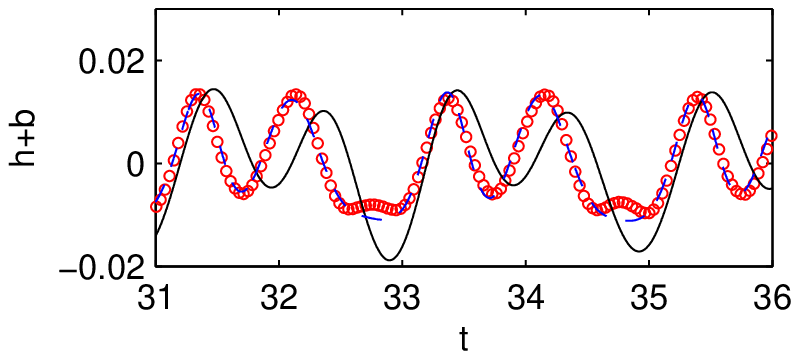}
\includegraphics[height=3.5cm,width=8.0cm,angle=0]{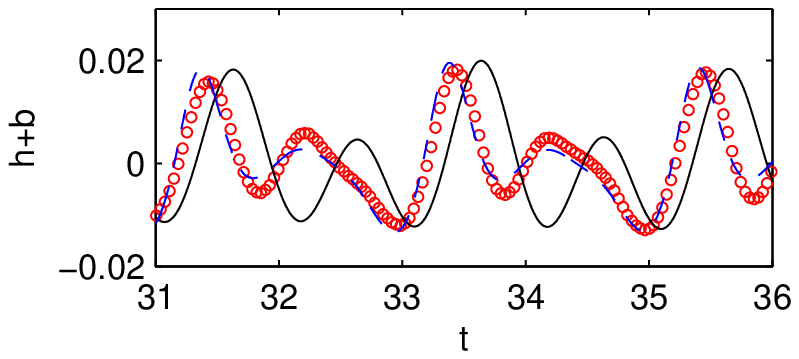}
\caption{Time series of surface elevations for waves passing over a submerged bar at $x=15.7$, $17.3$, $19$ and $21$ (from left to right, from top to bottom).  Circles: experimental data (\cite{Ding1994}), green solid line: numerical results with $\alpha=1$, blue solid line: numerical results with $\alpha=1.159$.}
\label{Fig:harmonics4}
\end{center}
\end{figure}

In this test, we further study the dispersive effect of the modified Green-Naghdi model. We consider a periodic wave propagation over a submerged bar with elliptic slope (see Figure \ref{Fig:elliptic-slope}), which is also given by
\begin{equation}\label{Eq:elliptic-slope}
b(x,y)=\left\{\begin{array}{lcl}
-0.1, &  r < \frac{47}{576} , \\
1.2 \sqrt{1-r}-1.25 ,        &  \frac{47}{576} \le r  \le \frac{287}{576},\\
-0.4,  & \mbox{elsewhere}~,
\end{array}\right.
\end{equation}
with $r=\frac{(x-12.5)^2}{100}+\frac{y^2}{16}$.

\begin{figure}
\begin{center}
\includegraphics[height=6.0cm,width=12.0cm,angle=0]{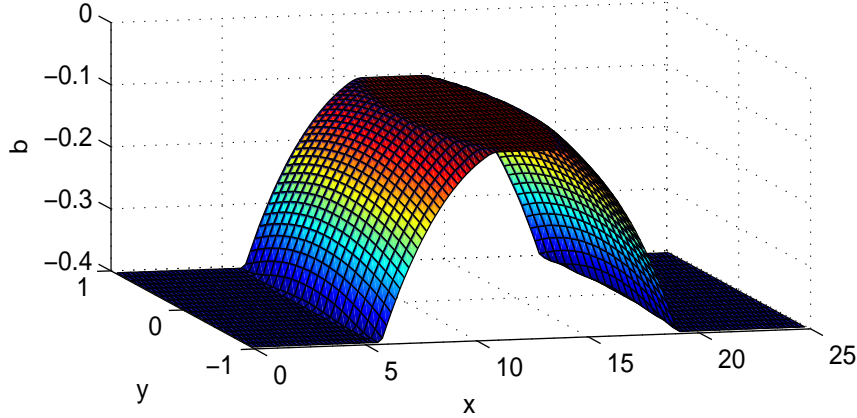}
\caption{The bottom topography for periodic wave propagation over a submerged bar with elliptic slope.}
\label{Fig:elliptic-slope}
\end{center}
\end{figure}

The computational domain is $[0,25]\times[-1,1]$, which is discretized into $125\times20$ uniform elements. At initial time, $h+b=0$ and $u=v=0$ in the computational domain. The incident wave (entering from the left) is a third-order Stokes wave given in \eqref{harmonics-incident} with $(T_0, a_0, \lambda) = (3, 0.01, 3.73)$. Solid wall boundary conditions are used at the top and bottom boundaries, an absorbing boundary condition is used at right boundary. The numerical free surfaces at $t=30$ for both models ($\alpha=1$ and $\alpha=1.159$) are shown in Figure \ref{Fig:Surface-elliptic-slope}. Overall, both numerical surfaces compare well with each other. To observe the discrepancies more clearly, the time series of numerical surfaces at several positions ($A_1(8,0)$, $A_2(9,0.5)$, $A_3(21,0)$, $A_4(18,0)$, $A_5(19,0.5)$ and $A_6(22,0.5)$) are shown in Figure \ref{Fig:Surface-elliptic-slope-time}. It can be seen that both numerical surfaces compare well with each other for gauges before the crest of the bar ($x \le 9$). However, the discrepancies in amplitude and phase can be observed for gauges beyond the crest of the bar ($x \ge 18$). These discrepancies again should be attributed to the dispersive character of the models.

\begin{figure}
\begin{center}
\includegraphics[height=9.5cm,width=12.0cm,angle=0]{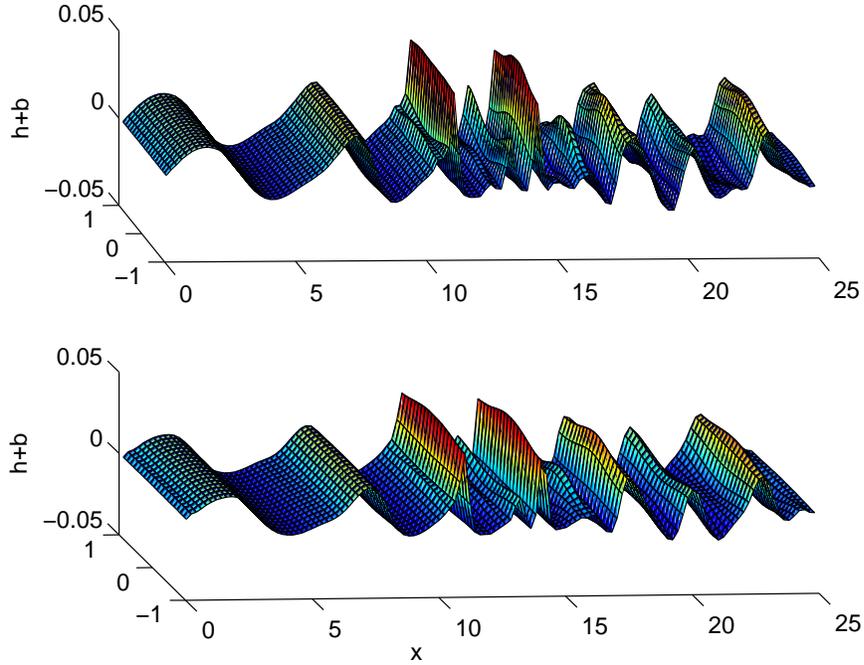}
\caption{Numerical surface at $t=30$ for periodic wave propagation over a submerged bar with elliptic slope. Top: numerical results with $\alpha=1$; bottom: numerical results with $\alpha=1.159$. }
\label{Fig:Surface-elliptic-slope}
\end{center}
\end{figure}

\begin{figure}
\begin{center}
\includegraphics[height=3.1cm,width=8.0cm,angle=0]{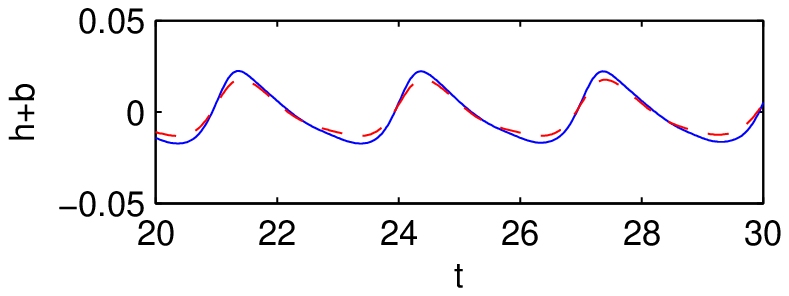}
\includegraphics[height=3.1cm,width=8.0cm,angle=0]{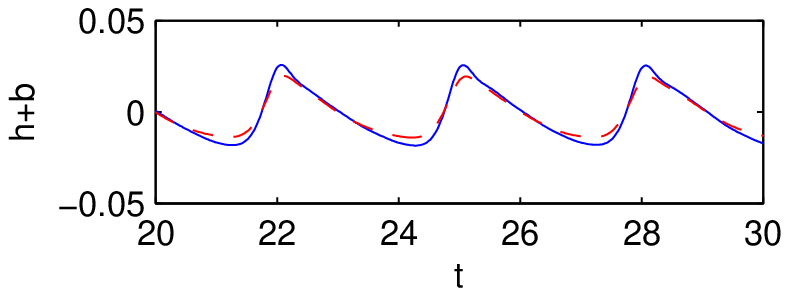}\\
\includegraphics[height=3.1cm,width=8.0cm,angle=0]{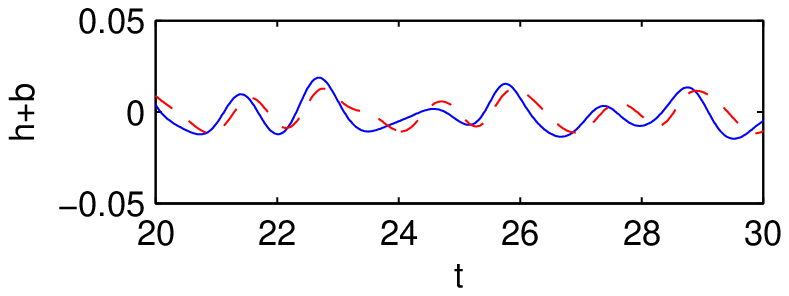}
\includegraphics[height=3.1cm,width=8.0cm,angle=0]{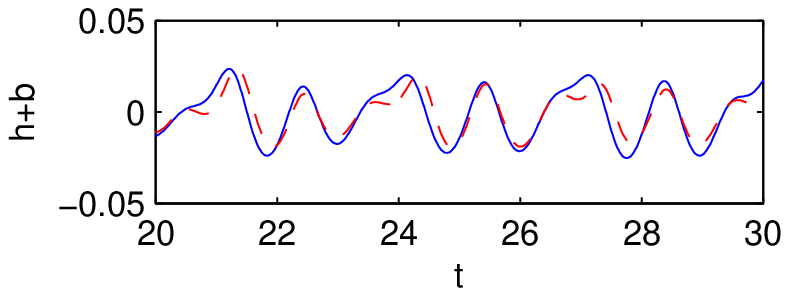}\\
\includegraphics[height=3.1cm,width=8.0cm,angle=0]{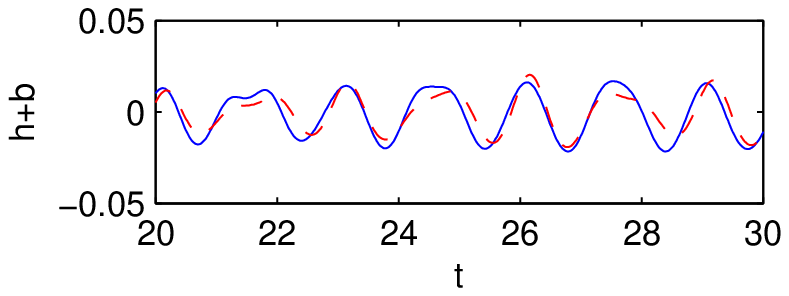}
\includegraphics[height=3.1cm,width=8.0cm,angle=0]{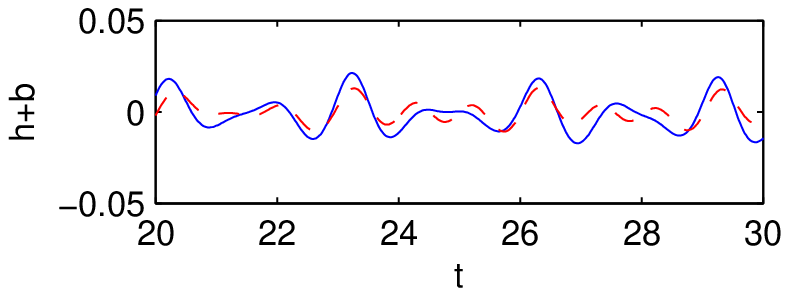}
\caption{Time series of numerical surface at six gages ($A_1(8,0)$, $A_2(9,0.5)$, $A_3(21,0)$, $A_4(18,0)$, $A_5(19,0.5)$ and $A_6(22,0.5)$, from left to right, from top to bottom) for periodic wave propagation over a submerged bar with elliptic slope. Blue lines: numerical results with $\alpha=1$; red dashed lines: numerical results with $\alpha=1.159$. }
\label{Fig:Surface-elliptic-slope-time}
\end{center}
\end{figure}

\subsection{Solitary wave propagation over a composite beach}
\label{sec:4.5}
To further investigate the robustness of the modified Green-Naghdi model, we simulate the propagation of solitary waves over a composite beach, which consists of three piece-wise linear segments, terminated with a vertical wall on the left. The slopes of the topography are defined as follows (\cite{Lannes2015}):
\begin{equation}\label{Eq:compo-beach}
s(x,y)=\left\{\begin{array}{lcl}
0,    & x\leq 15.04, \\
1/53, &  15.04 \le x \leq 19.4,\\
1/150, & 19.4 \le x \leq 22.33,\\
1/13,  & 22.33 \le x \leq 23.23.
\end{array}\right.
\end{equation}

We consider the propagation of a solitary wave: $h_1=0.22,h_2=1.73h_1$. The initial solitary waves, produced by using \eqref{Eq:so}, are located at $x=0$ and propagate to the right. The computational domain is $[-5,23.23]\times[-0.2,0.2]$ discretized into $500 \times 8$ uniform elements. The bottom along with $y=0$ and the initial solitary wave is shown in Figure \ref{Fig:composite_bottom}.
The outgoing boundary condition is used at the left boundary, the reflective boundary condition is employed at the right boundary and the periodic boundary condition is used at the upper and bottom boundaries. We observe the propagation over the beach, reflection on the vertical wall before traveling back to the left boundary. The time series of the wave elevation at several positions of gauges ($x=15.04, 19.4, 22.33$, $y=0$) are shown in Figure \ref{Fig:composite-beach-gages}. Overall, both sets of the numerical results obtained from the Green-Naghdi models ($\alpha=1$ and $\alpha=1.159$) compare well together before the wave encounters the vertical wall. However, discrepancies in amplitude and phase can be observed after the wave reflects on the vertical wall. Especially, discrepancies are clearer when the positions of gauges are far away from the vertical wall. Our numerical solution with $\alpha=1.159$ matches well with the one reported in \cite{Lannes2015}.

\begin{figure}
\begin{center}
\includegraphics[height=6.0cm,width=10.0cm,angle=0]{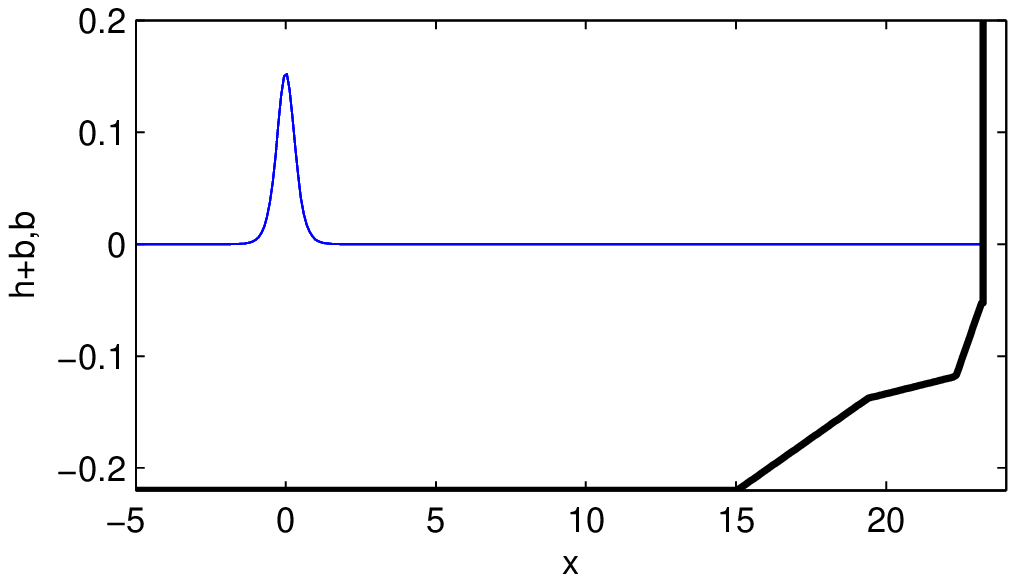}
\caption{The initial solitary wave and the bottom topography for solitary wave propagation over a composite beach.}
\label{Fig:composite_bottom}
\end{center}
\end{figure}

\begin{figure}
\begin{center}
\includegraphics[height=3.5cm,width=14.0cm,angle=0]{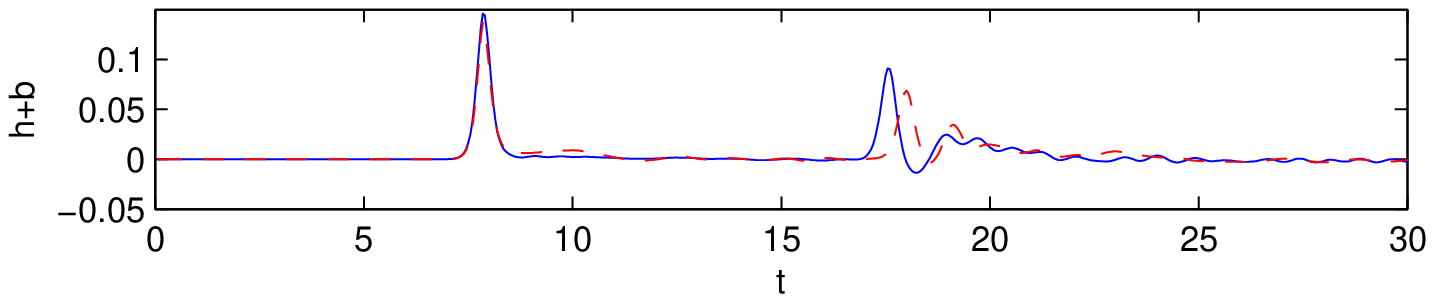}\\
\includegraphics[height=3.5cm,width=14.0cm,angle=0]{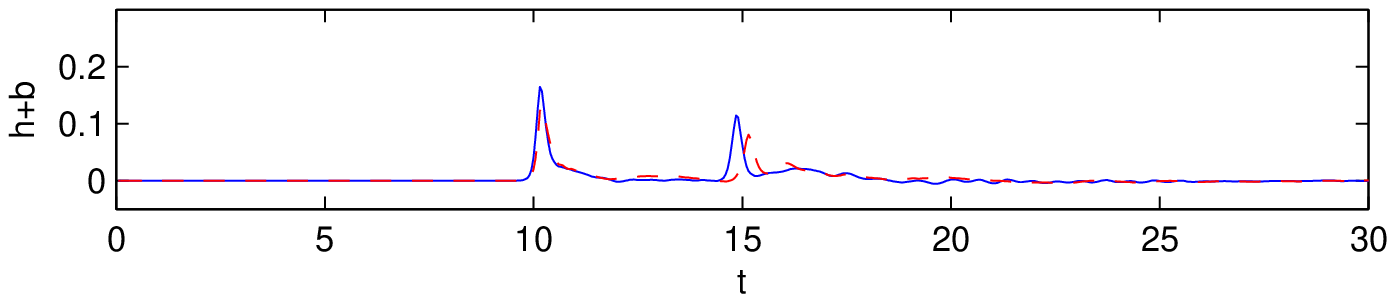}\\
\includegraphics[height=3.5cm,width=14.0cm,angle=0]{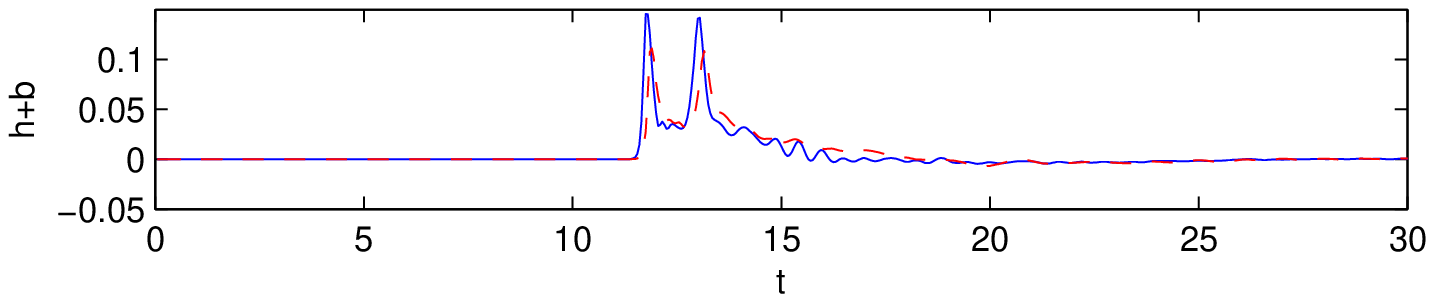}
\caption{The time series of the wave elevation at several gages ($x=15.04, 19.4, 22.33$, $y=0$, from top to bottom).  Blue line: numerical results with $\alpha=1$; Red dashed line: numerical results with $\alpha=1.159$.}
\label{Fig:composite-beach-gages}
\end{center}
\end{figure}

\section{Conclusions}
In this work, we derive a Green-Naghdi model with enhanced dispersive property and then develop a family of high order positivity-preserving and well-balanced numerical methods for its numerical solutions. These methods are based on the reformulation of the original system into a pseudo-conservation law coupled with an elliptic system.  Numerical experiments are presented to demonstrate the accuracy and achievement of expected  properties for the proposed schemes, and the capability of the Green-Naghdi equations to model a wide range of shallow water wave phenomena. Application of the proposed schemes to Green-Naghdi models taking into account the irrotational effect and the Coriolis effect due to the Earth's rotation, and proposal of fast solvers for the elliptic  part of the model will be envisioned for our future work.

\section*{Acknowledgments}
Maojun Li is partially supported by NSFC (Grant Nos. 11501062, 11701055, 11871139). Liwei Xu is partially supported by a Key Project of the Major Research Plan of NSFC (Grant No. 91630205) and a NSFC (Grant No. 11771068).

\section*{Appendix}

Here, we derive the fully nonlinear and weakly dispersive shallow water equations over the non-flat bottom in 2D space \eqref{GN-NFB-2D} along the same lines in \cite{Su1969} where the authors derived the one-dimensional equations on the flat bottom. Let $\Omega(t) $ be the domain in $R^3$ occupied by the water at time $t$. The propagation of the water is described by the fully nonlinear Euler equations
\begin{eqnarray}
% \nonumber to remove numbering (before each equation)
  \tilde u_t + \tilde u \tilde u_x + \tilde v \tilde u_y + \tilde w \tilde u_z &=& -p_x ~,   \label{Euler1}\\
  \tilde v_t + \tilde u \tilde v_x + \tilde v \tilde v_y + \tilde w \tilde v_z &=& -p_y ~,   \label{Euler2}\\
  \tilde w_t + \tilde u \tilde w_x + \tilde v \tilde w_y + \tilde w \tilde w_z &=& -p_z-g ~, \label{Euler3}
\end{eqnarray}
and the continuity equation
\begin{equation}\label{cont-equa}
\tilde u_x + \tilde v_y + \tilde w_z = 0 ~,
\end{equation}
where $(\tilde u, \tilde v, \tilde w)$ denotes the velocity of the water, $p$ is the pressure and $g$ is the gravitational constant.  The subscript $t$ denotes the partial derivative with respect to the time variable, $x$, $y$, and $z$ denote the partial derivatives with respect to the space variables. We assume that the density is taken as one. The boundary conditions are given by

(BC1) the kinematic condition at the free surface
\begin{equation}\label{boun-cond-1}
\tilde{w}^{(s)} = h_t + \tilde{u}^{(s)} (h+b)_x + \tilde{v}^{(s)} (h+b)_y~,
\end{equation}

(BC2) the impermeability of the bottom
\begin{equation}\label{boun-cond-2}
\tilde{w}^{(b)} = \tilde{u}^{(b)} b_x + \tilde{v}^{(b)} b_y~,
\end{equation}
where $h(x,y,t)$ denotes the depth of the water and $b(x,y)$ is the bottom topography.  The superscript $(s)$ and $(b)$  denote the quantities evaluated at the free surface and the bottom, respectively.

Let $u(x,y,t)$ and $v(x,y,t)$ denote the vertically averaged horizontal velocity in the x- and y-directions,  respectively, and be defined by
\begin{eqnarray}
u(x,y,t)&=&\frac{1}{h(x,y,t)}\int_{b}^{h+b}\tilde{u}(x,y,z,t)dz \label{aver-velo-u},\\
v(x,y,t)&=&\frac{1}{h(x,y,t)}\int_{b}^{h+b}\tilde{v}(x,y,z,t)dz. \label{aver-velo-v}
\end{eqnarray}

Along the same lines as in \cite{Su1969}, we want to find a set of equations which governs the evolution of the water depth $h(x,y,t)$, the average velocity $u(x,y,t)$ and $v(x,y,t)$ under the following assumptions:

(A1) we assume that the vertical movement of a particle is small compared with the horizontal movement, that is, we use the shallow water hypothesis, so that we can write
\begin{equation}\label{assu-1}
\tilde u(x,y,z,t) \simeq u(x,y,t) \mbox{ and } \tilde v(x,y,z,t) \simeq v(x,y,t),
\end{equation}

(A2) the dynamic condition at the free surface is assumed to be
\begin{equation}\label{assu-2}
p^{(s)} \simeq \mbox{constant}.
\end{equation}

Integrating (\ref{cont-equa}) with respect to $z$ gives
\begin{equation}\label{tilde-w}
\tilde{w}=\tilde{w}^{(b)}-\int_{b}^{z} \tilde{u}_x dz - \int_{b}^{z} \tilde{v}_y dz=-\left( \int_{b}^{z} \tilde{u} dz \right)_x - \left(\int_{b}^{z} \tilde{v} dz\right)_y~.
\end{equation}
Here,  we have used \eqref{boun-cond-2}. Equation \eqref{tilde-w} also implies
\begin{equation}\label{tilde-ws}
\tilde{w}^{(s)}=\tilde{w}^{(b)}-\int_{b}^{h+b} \tilde{u}_x dz - \int_{b}^{h+b} \tilde{v}_y dz.
\end{equation}

Multiplying \eqref{aver-velo-u} by $h$ and taking the derivative with respect to $x$, multiplying \eqref{aver-velo-v} by $h$ and taking the derivative with respect to $y$, and then adding them together, one arrives at
\begin{eqnarray}\label{hux-hvy}
(hu)_x+(hv)_y
&=&\left( \int_{b}^{h+b}\tilde{u}dz \right)_x + \left( \int_{b}^{h+b}\tilde{v}dz \right)_y \nonumber \\
&=& \int_{b}^{h+b}\tilde{u}_xdz + \tilde{u}^{(s)}(h+b)_x - \tilde{u}^{(b)} b_x \nonumber \\
&+& \int_{b}^{h+b}\tilde{v}_ydz + \tilde{v}^{(s)}(h+b)_y - \tilde{v}^{(b)} b_y.
\end{eqnarray}

Utilizing \eqref{boun-cond-1}, \eqref{boun-cond-2} and \eqref{tilde-ws}, one gets from \eqref{hux-hvy}
\begin{equation}\label{GN2D-1}
h_t+(hu)_x+(hv)_y=0.
\end{equation}
This equation gives the evolution of $h$ provided we know the evolution of $hu$ and $hv$.

Integrating \eqref{Euler1} with respect to $z$ from $z=b$ to $z=h+b$, performing an integration by parts for the fourth term and utilizing \eqref{cont-equa}, \eqref{boun-cond-1} and \eqref{boun-cond-2}, one obtains
\begin{equation}\label{hut1}
(hu)_t+\left( \int_{b}^{h+b}(\tilde{u}^2+p)dz\right)_x+\left( \int_{b}^{h+b}(\tilde{u} \tilde{v})dz\right)_y=p^{(s)}(h+b)_x-p^{(b)}b_x.
\end{equation}
Similarly, one has from \eqref{Euler2}
\begin{equation}\label{hvt1}
(hv)_t+\left( \int_{b}^{h+b}(\tilde{u} \tilde{v})dz\right)_x+\left( \int_{b}^{h+b}(\tilde{v}^2+p)dz\right)_y=p^{(s)}(h+b)_y-p^{(b)}b_y.
\end{equation}

Integrating \eqref{Euler3} with respect to $z$ from $z$ to $h+b$ yields
\begin{equation}\label{pres1}
p = \int_{z}^{h+b}( \tilde w_t + \tilde u \tilde w_x + \tilde v \tilde w_y + \tilde w \tilde w_z)dz + p^{(s)} + g(h+b-z),
\end{equation}
thus one gets
\begin{eqnarray}\label{aver-p}
\int_{b}^{h+b}pdz &=&\int_{b}^{h+b} \left( \int_{z}^{h+b}( \tilde w_t + \tilde u \tilde w_x + \tilde v \tilde w_y + \tilde w \tilde w_z)dz + p^{(s)} + g(h+b-z) \right)dz \nonumber \\
&=& \int_{b}^{h+b} (z-b)( \tilde w_t + \tilde u \tilde w_x + \tilde v \tilde w_y + \tilde w \tilde w_z ) dz + p^{(s)}h + \frac{1}{2}gh^2  \nonumber \\
&=& -\int_{b}^{h+b} (z-b)\left( \left( \int_{b}^{z} \tilde{u} dz \right)_{xt} + \left(\int_{b}^{z} \tilde{v} dz\right)_{yt} \right)dz \nonumber \\
&-& \int_{b}^{h+b} (z-b)\tilde{u} \left( \left( \int_{b}^{z} \tilde{u} dz \right)_{xx} + \left(\int_{b}^{z} \tilde{v} dz\right)_{yx} \right)dz \nonumber \\
&-& \int_{b}^{h+b} (z-b)\tilde{v} \left( \left( \int_{b}^{z} \tilde{u} dz \right)_{xy} + \left(\int_{b}^{z} \tilde{v} dz\right)_{yy} \right)dz \nonumber \\
&+& \int_{b}^{h+b} (z-b)(\tilde{u}_x+\tilde{v}_y) \left( \left( \int_{b}^{z} \tilde{u} dz \right)_{x} + \left(\int_{b}^{z} \tilde{v} dz\right)_{y} \right)dz \nonumber \\
&+&  p^{(s)}h + \frac{1}{2}gh^2~.
\end{eqnarray}
Here,  we have used \eqref{tilde-w} and \eqref{cont-equa}.

Now, substituting \eqref{aver-p} into \eqref{hut1} and \eqref{hvt1}, respectively, and using the assumptions \eqref{assu-1}-\eqref{assu-2}, one obtains
\begin{equation}\label{GN2D-2}
(hu)_t+\left( hu^2+ \frac{1}{2}gh^2+\frac{1}{3}h^3 \Phi + \frac{1}{2} h^2 \Psi \right)_x + \left( huv \right)_y =-\left(gh+\frac{1}{2}h^2\Phi+h\Psi\right)b_x
\end{equation}
and
\begin{equation}\label{GN2D-3}
(hv)_t+\left(huv\right)_x+\left( hv^2+ \frac{1}{2}gh^2+\frac{1}{3}h^3\Phi+\frac{1}{2}h^2\Psi\right)_y=-\left(gh+\frac{1}{2}h^2\Phi+h\Psi\right)b_y
\end{equation}
where
\begin{eqnarray}
\Phi &=& - u_{xt}-u u_{xx}+ u_x^2 - v_{yt}-v v_{yy}+ v_y^2 - uv_{xy}- u_{xy}v+ 2 u_x v_y, \label{phi2D-1} \\
\Psi &=& b_x u_{t}+b_x u  u_{x}+b_{xx} u^2 + b_y v_{t} + b_y v v_y +b_{yy} v^2 + b_y u v_x + b_x u_y v + 2b_{xy}uv. \label{psi2D-1}
\end{eqnarray}

Therefore, we get a set of equations by combining \eqref{GN2D-1}, \eqref{GN2D-2} and \eqref{GN2D-3}
\begin{equation}\label{GN-NFB-2D-1}
\left\{\begin{array}{lclcl}
h_t+(hu)_x+(hv)_y=0,\\
(hu)_t+\left( hu^2+ \frac{1}{2}gh^2+\frac{1}{3}h^3 \Phi + \frac{1}{2} h^2 \Psi \right)_x + \left( huv \right)_y =-\left(gh+\frac{1}{2}h^2\Phi+h\Psi\right)b_x,\\
(hv)_t+\left(huv\right)_x+\left( hv^2+ \frac{1}{2}gh^2+\frac{1}{3}h^3\Phi+\frac{1}{2}h^2\Psi\right)_y=-\left(gh+\frac{1}{2}h^2\Phi+h\Psi\right)b_y.
\end{array} \right.
\end{equation}

%\section*{References}
\bibliographystyle{plainnat}

\end{document}